\documentclass{article}
\usepackage{amsmath}
\usepackage{amsthm}
\usepackage{amsfonts}
\usepackage{graphicx}
\usepackage{xcolor}
\usepackage[colorlinks=true,linkcolor=blue,citecolor=blue,%
        urlcolor=blue]{hyperref}

 \makeatletter
 \let\@fnsymbol\@arabic
 \makeatother

\newtheorem{thm}{Theorem}[section]
\newtheorem{corollary}[thm]{Corollary}
\newtheorem{lemma}[thm]{Lemma}
\newtheorem{prop}[thm]{Proposition}

\newtheorem{obs}[thm]{Observation}

 \let\ovl\overline
 \let\al\alpha
 \let\ga\gamma
 \let\ep\epsilon
 \let\Ga\Gamma
 \let\Si\Sigma
 \def\mR{{\mathbb R}}
 \def\mZ{{\mathbb Z}}
 \def\cN{{\fam0 N}}
 \def\cS{{\fam0 S}}
 \let\vP=P
 \let\vQ=Q
 \let\vX=X
 \let\vY=Y
 \def\double#1{{#1}[\,\ovl{K_2}\,]}
 \long\def\ignore#1{}
 \let\eH\Theta
 \let\eJ\Psi
 \let\diamsum\Diamond
 \long\def\ignore#1{}
 \long\def\noignore#1{#1}
 \def\hnull{\leavevmode\vrule height0pt depth0pt width0pt}
 \def\og{g}
 \let\tw\widetilde
 \def\ng{{\tw g}}
 \def\even{_{\fam0 even}}
 \let\ok\Theta
 \def\nk{\tw\Theta}
 \let\ec\varepsilon
 
 \let\ab\allowbreak
 \newdimen\rotsysstrutheight \rotsysstrutheight=18pt
 \newdimen\rotsysfinalskip   \rotsysfinalskip=5pt
 \def\rotsys#1{%
  \raise\rotsysstrutheight\hbox{\vtop{%
   \hrule height0pt depth0pt width1pt
   \hbox{\normalfont%
    \tabcolsep=2pt
    \def\arraystretch{0.8}%
    \begin{tabular}{lllllllllllllllll}%
     \vrule height\rotsysstrutheight depth0pt width0pt
     #1
     \noalign{\vskip\rotsysfinalskip}
    \end{tabular}
  }}}
 }
 \definecolor{dkred}{rgb}{0.7,0,0}
 \def\shortversion{%
    \let\shortonly\noignore
    \let\longonly\ignore
    \long\def\shortlong##1##2{{}##1}%
 }
 \def\longversion{%
    \let\shortonly\ignore
    \long\def\longonly##1{{{}##1}}%
    \long\def\shortlong##1##2{{}##2}%
 }
 \def\longversionred{%
    \let\shortonly\ignore
    \long\def\longonly##1{{\color{dkred}{}##1}}%
    \long\def\shortlong##1##2{{\color{dkred}{}##2}}%
 }
 
 \shortversion

\longversion

\begin{document}
 \title{Quadrangular embeddings of complete graphs and the Even Map
Color Theorem\longonly{\ (with details)}%
 \thanks{%
 The United States Government is authorized to reproduce and distribute
reprints notwithstanding any copyright notation herein.}}
 \author{
 Wenzhong Liu%
 \thanks{Department of Mathematics, Nanjing University of
Aeronautics and Astronautics, Nanjing 210016, China. Email:
\texttt{wzhliu7502@nuaa.edu.cn}.},\;\;
 Serge~Lawrencenko\thanks{
Russian State University of Tourism and Service,
Institute for Tourism and Hospitality,
Bldg. 32A, Kronstadt Boulevard, Moscow, 125438, Russia.
 Email: \texttt{lawrencenko@hotmail.com}.
 },\;\;
 Beifang~Chen\thanks{
Hong Kong University of Science and Technology,
Clear Water Bay, Kowloon, Hong Kong, China.
 Email: \texttt{mabfchen@ust.hk}.
 },\;\;
 M.~N.~Ellingham\thanks{
Department of Mathematics, Vanderbilt University, Nashville, TN 37240,
USA, Email: \texttt{mark.ellingham@vanderbilt.edu}.%
 },\\
 Nora Hartsfield\thanks{
Deceased.
},\;
 Hui Yang\thanks{
Department of Mathematics,
Guizhou University, Guiyang, 550025, China.
 Email: \texttt{hui-yang@163.com}.
 },\;\;
 Dong Ye\thanks{Department of Mathematical Sciences, Middle Tennessee
State University, Murfreesboro, TN 37132, USA,
Email: \texttt{dong.ye@mtsu.edu}.}\;\;  and
 Xiaoya Zha\thanks{Department of Mathematical Sciences, Middle Tennessee
State University, Murfreesboro, TN 37132, USA,
Email: \texttt{xiaoya.zha@mtsu.edu}.}
 }

\date{2 June 2016; revised 19 October 2018 and 22 July 2021}
\maketitle

\begin{abstract}
 Hartsfield and Ringel constructed orientable quadrangular embeddings of
the complete graph $K_n$ for $n\equiv 5 \pmod 8$, and nonorientable ones
for $n \ge 9$ and $n\equiv 1 \pmod 4$.
 These provide minimal quadrangulations of their underlying surfaces.
 We extend these results to determine, for every complete graph $K_n$,
$n \ge 4$, the minimum genus, both orientable and
nonorientable, for the surface in which $K_n$ has an embedding with all
faces of degree at least $4$, and also for the surface in which $K_n$
has an embedding with all faces of even degree.
 These last embeddings provide sharpness examples for a result of
Hutchinson bounding the chromatic number of graphs embedded with all
faces of even degree, completing the proof of the Even Map Color Theorem.
 We also show that if a connected simple graph $G$ has a perfect
matching and a cycle then the lexicographic product $G[K_4]$ has
orientable and nonorientable quadrangular embeddings; this provides new
examples of minimal quadrangulations.

 \medskip


\end{abstract}


 \section{Main results}
 \label{sec:mainres}

 In this paper surfaces are connected compact $2$-manifolds without
boundary.  
 The orientable surface of genus $h$ is denoted $S_h$, and the
nonorientable surface of genus $k$ is denoted $N_k$.
 The Euler characteristic of a surface $\Si$ is denoted $\ec(\Si)$,
which is $2-2h$ for $S_h$, and $2-k$ for $N_k$.

 Frequently we want embeddings of a given graph with minimum
genus, which have faces that are small, often triangular faces.
 In particular, the determination of the minimum genus of complete
graphs as part of the Map Color Theorem \cite{R} was one of the driving
forces behind the development of topological graph theory.
 However, we can also consider the minimum genus of embeddings with
restrictions on face degrees.  In this paper we consider embeddings
where all faces have degree at least $4$, or all faces have even degree.
 Euler's formula and face/edge counting imply the following.

 \begin{obs}\label{obs:quadedges}
 If $\Phi$ is an embedding of an $n$-vertex $m$-edge graph in a surface
$\Si$ with all faces of degree at least $4$, then $m \le 2n-2\ec(\Si)$,
with equality if and only if the embedding is cellular and every face
degree is $4$.  For a complete graph $K_n$, $n \ge 4$, such an embedding
has $n(n-5) \le -4\ec(\Si)$, with equality if and only if the embedding
is quadrangular.
 \end{obs}

 In this paper we completely resolve the question of the minimum genus
of a surface in which $K_n$ has an embedding with all faces of degree at
least $4$, or with all faces of even degree.  These results also
complete the proof of a coloring result, the Even Map Color Theorem.  
 In 1975 Hutchinson \cite{H75} showed that the chromatic number bound of
the Map Color Theorem can be significantly improved for even-faced
embeddings; our results improve her bound in one case and provide
sharpness examples.
 We also provide some constructions for minimal quadrangulations, simple
quadrangulations with a minimum number of vertices in a given surface.

 Our main results are as follows (see Section \ref{sec:prelim} for
definitions not stated here).
 For a graph $G$ and positive integer $d$, the {\em orientable $d$-genus
$\og_d(G)$} and the {\em orientable even-faced genus $\og\even(G)$} are
the smallest $h \ge 0$ for which $G$ has a cellular embedding in $S_h$
with all face degrees at least $d$, or with all face degrees even,
respectively.  We can similarly define the {\em nonorientable $d$-genus
$\ng_d(G)$} and the {\em nonorientable even-faced genus $\ng\even(G)$}
(for convenience we take $\ng_d(G)$ or $\ng\even(G)$ to be $0$ if $G$
has a suitable planar embedding).

 \begin{thm}\label{thm:genusori}
 Let $f(n) = 1 + \lceil n(n-5)/8\rceil$.  Then
 $$\og_4(K_n) = f(n) \text{\quad if $n \ge 4$,}
 \text{\qquad and\qquad}
 \og\even(K_n) = 
    \begin{cases}
        f(n) & \text{if $n \ge 4$ and $n \ne 6$,} \\
        f(6)+1=3 & \text{if $n=6$.} \\
 \end{cases}$$
 For $n = 5$, for $n \ge 7$, and for $\og_4(K_6)$ there is a
face-simple closed-2-cell embedding of $K_n$ realizing each equation.
 Such an embedding is quadrangular if and only if $n \equiv 0$ or $5
\pmod 8$.
 \end{thm}

 \begin{thm}\label{thm:genusnon}
 Let ${\tw f}(n) = 2 + \lceil n(n-5)/4\rceil$.  Then
 $$\ng_4(K_n) = \ng\even(K_n) = 
    \begin{cases}
        {\tw f}(n) & \text{if $n \ge 4$ and $n \ne 5$,} \\
        {\tw f}(5)+1=3 & \text{if $n=5$.} \\
 \end{cases}$$
 For $n = 4$ and for $n \ge 6$ there is a closed-2-cell
embedding of $K_n$ realizing this pair of equations, that is face-simple
if $n \ge 6$.
 Such an embedding is quadrangular if and only if $n \equiv 0$ or $1
\pmod 4$.
 \end{thm}

 Let $\chi(\Phi)$ and $\chi^*(\Phi)$ denote the number of colors needed
to properly vertex-color or face-color, respectively, a graph embedding
$\Phi$.  We will ignore loops when vertex-coloring and {\em monofacial}
edges (with the same face on both sides) when face-coloring, so $\chi$
and $\chi^*$ are defined for all embeddings.  

 \begin{thm}[Even Map Color Theorem]\label{thm:evenmct}
 Let $\Phi$ be a (not necessarily cellular) embedding of a (not
necessarily connected) graph (loops and multiple edges allowed) in a
surface $\Si$.  Define
 $$H\even(\Si) = \left\lfloor \frac{5 + \sqrt{25-16\ec(\Si)} }{2}
		\right\rfloor
 \;\text{if $\Si \ne S_0$,}\quad\text{and}\quad
 c(\Si) = \begin{cases}
	2 & \text{if $\Sigma = S_0$,} \\
	H\even(\Si)-1 & \text{if $\Sigma = N_2$ or $S_2$,} \\
	H\even(\Si) & \text{otherwise.} \\
 \end{cases}
 $$

 \noindent
 (a) If every face of $\Phi$ has even degree (individual face boundary
components may have odd length), then (ignoring loops when coloring)
$\chi(\Phi) \le c(\Si)$.

 \noindent
 (b) If every vertex of $\Phi$ has even degree, then (ignoring
monofacial edges when coloring) $\chi^*(\Phi) \le c(\Si)$.

 Moreover, for every surface $\Si$ there exist face-simple
closed-$2$-cell embeddings of connected simple graphs, that are
quadrangular for (a) and $4$-regular for (b), which show that these
bounds are sharp.
 \end{thm}

 The following provides new constructions of minimal quadrangulations,
as well as giving an alternative proof of some cases of Theorems
\ref{thm:genusori} and \ref{thm:genusnon}.

 \begin{thm}\label{thm:compk4}
 Let $G$ be a connected simple graph with a perfect matching.  Then
$G[K_4]$ has a face-simple orientable quadrangular embedding.  Moreover,
if $G$ also has a cycle, then $G[K_4]$ also has a face-simple
nonorientable quadrangular embedding.
 \end{thm}

 Section \ref{sec:backgd} provides some background to our results, and
Section \ref{sec:prelim} provides precise definitions and preliminary
results.  Section \ref{sec:embdiamsum} proves Theorems
\ref{thm:genusori} and \ref{thm:genusnon}, and Section \ref{sec:color}
proves the Even Map Color Theorem.  Section \ref{sec:embgsvg} proves
Theorem \ref{thm:compk4}, and Section \ref{sec:minquad} shows that
results from Sections \ref{sec:embdiamsum} and \ref{sec:embgsvg} yield
minimal quadrangulations.  Section \ref{sec:conclusion} contains some
final remarks.

 \shortlong%
 {A version of this paper with some additional details has been posted
on the \texttt{arXiv} \cite{LLCEHYYZp}.}%
 {This version of this paper contains some details not included in the
published version
\cite{LLCEHYYZ}.}

 \section{Background}
 \label{sec:backgd}

 The minimum genus of the complete graph $K_n$ and conditions for the
existence of triangular embeddings of $K_n$ were determined as part of
the well-known Map Color Theorem \cite{R}, which extended the Four
Color Theorem to other surfaces.
 Subsequently there were a number of results showing existence of
multiple triangular embeddings of certain complete graphs,
such as \cite{AB92, LNW, Y70}, and then providing lower bounds on the
number of nonisomorphic triangular embeddings of $K_n$ for certain
families of $n$, such as \cite{BGGS, GGS, GK, KV}.
 %

 Less work has been done on quadrangular embeddings of complete graphs,
or embeddings of complete graphs with all faces of degree at least $4$,
or all faces of even degree.  By Observation \ref{obs:quadedges},
$n(n-5)=-4\ec(\Si)$ when there is a quadrangulation of $K_n$ in $\Si$,
which means that $n \equiv 0$ or $5 \pmod 8$ in the orientable case, and
$n \equiv 0$ or $1 \pmod 4$ in the nonorientable case.
 Hartsfield and Ringel \cite{HRori, HRnon} obtained the following
results, mostly using current graphs, covering half of the possible
values of $n$ for which quadrangular embeddings of $K_n$ may exist.

 \begin{thm}[Hartsfield and Ringel \cite{HRori, HRnon}]\label{thm:HR}
 A complete graph $K_n$ with $n=8$ or $n\equiv 5\pmod 8$ has a face-simple orientable
quadrangular embedding.
 A complete graph $K_n$ with $n \ge 9$ and $n\equiv 1 \pmod 4$ has a
face-simple nonorientable quadrangular embedding.  However,
$K_5$ has no nonorientable quadrangular embedding.
 \end{thm}

 The embeddings in Theorem \ref{thm:HR} are minimal quadrangulations.
 Hartsfield and Ringel also constructed quadrangular embeddings of the
generalized octahedron $O_{2k} = K_k[\overline{K_2}]$ that are minimal.
 We discuss minimal quadrangulations in more detail in
Section \ref{sec:minquad}.
 The fact that $K_5$ has no nonorientable quadrangular embedding was
also proved earlier (in dual form) by Hutchinson \cite{H75}.

 Using current graphs, Korzhik and Voss \cite{KV} constructed
exponentially many nonisomorphic orientable quadrangular embeddings of
$K_{8s+5}$ for $s\ge 1$, and 
 Korzhik \cite{K} constructed superexponentially many nonisomorphic
orientable and nonorientable quadrangular embeddings of $K_{8s+5}$ for
$s \ge 2$.
 Grannell and McCourt \cite{GM} constructed many nonisomorphic
orientable embeddings of complete graphs $K_n$ with faces bounded by
$4k$-cycles for $k \ge 2$, when $n=8ks+4k+1$ for $s\ge 1$.

 It is natural to ask whether the results in Theorem \ref{thm:HR} can be
extended to the other cases where quadrangular embeddings of $K_n$ might
exist, namely $n \equiv 0 \pmod 8$ for orientable embeddings, and $n
\equiv 0 \pmod 4$ for nonorientable embeddings.
 When $K_n$ has a quadrangular embedding it is a minimal
quadrangulation, and realizes $\og_4(K_n)$ and $\og\even(K_n)$, or
$\ng_4(K_n)$ and $\ng\even(K_n)$.  But we can also try to determine
these parameters even if $K_n$ does not have a quadrangular embedding. 
Our Theorems \ref{thm:genusori} and \ref{thm:genusnon} resolve all
of these questions, and we provide new proofs for the existence results
in Theorem \ref{thm:HR}.

 Some explanation of the origins of this paper is appropriate.
 In the early 1990s one of us, Hartsfield, developed a technique for
constructing quadrangulations by ``adding handles using diagonals''.  
 She used this technique to derive a number of results on quadrangular
embeddings, including that $K_n$ has a nonorientable quadrangular
embedding when $n \equiv 0 \pmod 4$ \cite{Ha94}.  She also applied this
to derive results on $\ng_4(K_n)$ and $\ng\even(K_n)$ in a paper that
was submitted for publication in 1994 \cite{Ha94q}.
 As indicated in \cite{H95}, Hartsfield was aware that her results would
give sharpness examples for Hutchinson's coloring results \cite{H75}.
 Hartsfield's papers \cite{Ha94, Ha94q} outlined proofs (providing basis
cases and examples of inductive steps, such as from $K_8$ to $K_{16}$)
but did not give complete general arguments.

 In the late 1990s three of us, Chen, Lawrencenko and Yang (CLY),
derived results on $\og_4(K_n)$ using current graphs \cite{CLY98, LCY18}.
 These were submitted for publication in 1998.
 When Hartsfield and CLY discovered they had been working on similar
results, they decided to combine their results into a single paper. 
 Unfortunately, this single paper was never finished.  Some researchers
were aware of the results of Hartsfield (cited in \cite{H95})
and of CLY (cited in \cite{Su05}) but they were not publicly
available.

 Around 2015 the remaining four authors, Ellingham, Liu, Ye and Zha
(ELYZ), worked on some problems of Craft \cite{DC95} on quadrangular
embeddings of composition graphs.  ELYZ realized that their
constructions (see Section \ref{sec:embgsvg}) provided orientable
quadrangular embeddings for $K_n$ with $n \equiv 0 \pmod 8$, which did
not seem to be in the literature.  ELYZ also came up with a diamond
sum construction (see Section \ref{sec:embdiamsum}) for nonorientable
quadrangular embeddings of $K_n$ for $n \equiv 0 \pmod 4$.  ELYZ's
results were written up \cite{LEYZ16} and submitted in 2016.
 After submission of their paper ELYZ were informed of the earlier
unpublished results of Hartsfield and CLY.  It was decided to combine
all of the results into the present joint paper.  Although Nora
Hartsfield died in 2011 we think it is appropriate to include her as an
author.

 We hope that eventually the other proofs of Theorems \ref{thm:genusori}
and \ref{thm:genusnon} using Hartsfield's diagonal technique and current
graphs will also appear.
 For the current graph results, some modification of the index $2$
current graphs in \cite{CLY98, LCY18} is required, and we hope to
provide nonorientable constructions as well as orientable ones.
 A paper using a combination of current graphs and Hartsfield's diagonal
technique is in preparation \cite{LCYHp} and additional papers may
follow.

 \section{Preliminaries}
 \label{sec:prelim}

 \subsection{Graph embeddings}
 \label{ss:graphemb}

 Our graphs may have loops or multiple edges; {\em simple\/} graphs have
neither.
 We say a graph embedding has some graph property (such as
bipartiteness) if the underlying graph has this property.
 A face of a graph embedding is {\em cellular\/} if it is homeomorphic
to an open disk. 
 We often identify a cellular face by referring to its
bounding cycle or bounding closed walk.
 A graph embedding is {\em cellular\/} if every face is cellular, {\em
closed-$2$-cell} if it is cellular and every face is bounded by a cycle
(with no repeated vertices), and {\em face-simple} if every two distinct
faces share at most one boundary edge.
 In this paper all embeddings are cellular unless we specifically refer
to a {\em general} embedding, which means that faces may have multiple
boundary components and internal handles or crosscaps.

 Suppose $\Phi$ is a graph embedding in surface $\Si$.
 The {\em degree\/} of a face is the number of sides of edges with which
it is incident.
 A $k$-face is a face of degree $k$, and a {\em $C_k$-face\/} is a
cellular face bounded by a $k$-cycle.   A cellular $k$-face is bounded
by a single closed walk of length $k$, which may or may not be a
$k$-cycle.
 The minimum vertex degree and minimum face degree of $\Phi$ are denoted
$\delta(\Phi)$ and $\delta^*(\Phi)$, respectively.
 The embedding $\Phi$ is {\em even-vertexed} or {\em even-faced} if
every vertex or every face, respectively, has even degree.
 An even-faced noncellular embedding may have individual face boundary
component walks of odd length, as long as the overall degree of each
face is even.
 We say $\Phi$ is {\em quadrangular\/}, or a {\em quadrangulation of
$\Si$\/}, if every face is a $C_4$-face.
 We also refer to a $C_4$-face as a {\em quadrilateral}.
 A quadrangulation of $\Si$ is \emph{minimal} if its underlying graph is
simple and connected, and there is no quadrangular embedding of a simple
graph of smaller order in $\Si$.
 Similarly, a {\em triangulation} of a surface $\Si$ is an embedding of
a graph in $\Si$ such that every face is a $C_3$-face.

 \begin{lemma}[Euler's inequality]
 \label{lem:eulerineq}
 Suppose we have a general embedding of a graph $G$ in a surface of
Euler characteristic $\ec$, with $n$ vertices, $m$ edges and $r$ faces. 
Then $n - m + r \ge \ec$, with equality if and only if the embedding is
cellular,
 \end{lemma}

 \begin{obs}\label{obs:F}
 Suppose $\Phi$ is a general embedding of a simple graph $G$ and
$\delta(\Phi) \ge 2$.  Then every face of $G$ of degree at most $5$ is
bounded by a single cycle.
 \end{obs}

 \begin{obs}\label{obs:F2}
 Suppose $\Phi$ is a general even-faced embedding of a simple connected
graph on at least three vertices.
 Then every face boundary walk has length at least $3$, and hence
$\delta^*(\Phi) \ge 4$.
 \end{obs}

 \begin{obs}\label{obs:FS}
 Suppose $\Phi$ is a quadrangular embedding of a simple connected graph
and $\delta(\Phi) \ge 3$.
 If $\Phi$ is not face-simple then it contains two faces of the form
$(uvwx)$ and $(uvxw)$.  Thus, if $\Phi$ is orientable or bipartite then
it is face-simple.
 \end{obs}

 \subsection{Graph operations}

 Let $G$ and $H$ be simple graphs.
 The complement of $G$ is denoted $\ovl{G}$.
 The {\em composition} (or {\em
lexicographic product}) of $G$ and $H$, denoted $G[H]$, has vertex set
$V(G)\times V(H)$, with two vertices $(v_1, w_1)$ and $(v_2,w_2)$
adjacent if and only if either (i) $v_1v_2\in E(G)$ or (ii) $v_1=v_2$
and $w_1w_2\in E(H)$. For example,
 $K_n[K_2]$ is the complete graph $K_{2n}$,
 and $K_n[\,\ovl{K_2}\,]$ is the generalized octahedron
$O_{2n}=K_{2n}-nK_2$.
 The {\em join} of $G$ and $H$, denoted $G+H$, is the union of $G$ and
$H$ together with one edge $uv$ for each $u \in V(G)$ and $v \in V(H)$.
 For example, $K_4+K_n$ is the complete graph $K_{n+4}$.

 \subsection{The diamond sum}

Let $G$ and $G'$ be two simple graphs with embeddings $\Phi$ and $\Phi'$
in disjoint surfaces $\Si$ and $\Si'$, respectively.
 Suppose that $k \ge 1$ and both $G$ and $G'$ have a vertex of degree $k$,
say $v$ and $v'$ respectively. Let $v$ have neighbors $v_0, v_1, . . . ,
v_{k-1}$ in cyclic order around $v$ in $\Phi$, and let $v'$ have
neighbours $v_0', v_1',..., v_{k-1}'$ in cyclic order around $v'$ in
$\Phi'$. There is a closed disk $D$ that intersects $G$ in $v$ and the
edges $vv_0, vv_1,..., vv_{k-1}$, and so that the boundary of $D$
intersects $G$ at $v_0, v_1,..., v_{k-1}$.  Similarly, there is a closed
disk $D'$ that intersects $G'$ in $v'$ and the edges $v'v_0', v'v_1', .
. ., v'v_{k-1}'$ and so that the boundary of $D'$ intersects $G'$ at
$v_0', v_1',..., v_{k-1}'$. Remove the interiors of $D$ and $D'$, and
identify  their boundaries so that $v_i$ is identified with $v_i'$ for
$0\le i\le  k-1$.
 The resulting embedding is called a {\em diamond sum of $\Phi$ and
$\Phi'$ at $v$ and $v'$\/}, denoted $\Phi \diamsum_{v,v'} \Phi'$ or just
$\Phi \diamsum \Phi'$.  Its graph is denoted $G \diamsum G'$ and the
surface is the connected sum $\Si \# \Si'$.
 Note that $\Phi\diamsum
\Phi'$ is orientable if and only if both $\Phi$ and $\Phi'$ are orientable.

 The diamond sum was first used by Bouchet \cite{Bo78}
in dual form to derive a new proof of the minimum genus of $K_{m,n}$.
 Bouchet's construction was later reinterpreted in more general
situations in \cite{KSZ,MMP,MPP}.

 The diamond sum of two cellular embeddings is cellular.
 It is also not difficult to see that if $\Phi$ and $\Phi'$ are
quadrangular and the diamond sum $\Phi \diamsum \Phi'$ is simple, then
$\Phi \diamsum \Phi'$ is also quadrangular.
 To build embeddings in Section \ref{sec:embdiamsum} that are
face-simple and closed-2-cell, we rely on the following  technical
extension of this observation, which allows $\Phi'$ to contain
non-$C_4$-faces.

 \begin{lemma}\label{lem:diamsumquad}
 Suppose $\Phi$ is a face-simple quadrangular embedding of a simple
graph $G$, $\delta(\Phi) \ge 3$, $v \in V(G)$, and the neighbors of $v$
in $G$ are independent.
 Suppose $\Phi'$ is a closed-2-cell embedding of a simple graph $G'$,
$v' \in V(G')$, and every pair of distinct faces of $\Phi'$ shares at
most one edge of $G'-v'$.
 Then $\Phi'' = \Phi \diamsum_{v,v'} \Phi'$ is a face-simple closed-2-cell
embedding and there is a degree-preserving bijection between the
non-$C_4$-faces in $\Phi'$ and the non-$C_4$-faces in $\Phi''$.
 \end{lemma}

 Note that the condition on pairs of distinct faces of $\Phi'$ holds if
$\Phi'$ is face-simple.

 \begin{proof}
 There are three types of faces in $\Phi''$: (1) those that use only edges
of $G-v$; (2) those that use edges of both $G-v$ and $G'-v'$; and (3)
those that use only edges of $G'-v'$.
 Represent the faces of $\Phi$ using $v$ as $Z_i = (v v_i w_i v_{i+1})$
and the faces of $\Phi'$ using $v'$ as $Z_i' = (v' v_i' \ldots
v_{i+1}')$, for $0 \le i \le k-1$, taking subscripts modulo
$k$.  Since the neighbors of $v$ are independent, $w_i$ is not a
neighbor of $v$.

 All faces in $\Phi$ and $\Phi'$ are bounded by cycles since $\Phi$ and
$\Phi'$ are closed-2-cell, which implies that faces of type (1) and (3)
are bounded by cycles, and all $Z_i$ and $Z_i'$ are cycles.
 Thus, every face of type (2) is a face $Z_i''$ obtained by combining
paths $Z_i-v = v_i w_i v_{i+1}$ and $Z_i'-v'$ by identifying $v_i$ with
$v_i'$ and $v_{i+1}$ with $v_{i+1}'$.
 Since $w_i$ is not a neighbor of $v$, it is not identified
with any vertex of $Z_i'-v'$, so $Z_i''$ is a cycle, and of the same
length as $Z_i'$.  Thus, $\Phi''$ is closed-2-cell.

 Since all faces of type (1) are $C_4$-faces, mapping each
non-$C_4$-face $Z_i'$ to $Z_i''$ and each non-$C_4$-face of type (3) to
itself gives the required degree-preserving bijection for
non-$C_4$-faces.

 Let $m_{st}$ be the maximum number of edges shared by a face of type
($s$) and a distinct face of type ($t$).  Clearly $m_{13}=0$; since
$\Phi$ is face-simple, $m_{11}, m_{12} \le 1$; and by the hypothesis on
$\Phi'$, $m_{32}, m_{33} \le 1$.
 Consider two arbitrary distinct faces
$Z_i'', Z_j''$ of type (2). Suppose $Z_i-v$ and $Z_j-v$ share an edge
$ab$.  Since $\Phi$ is simple and $\delta(\Phi) \ge 3$, we cannot have
$Z_i=(vabc)$ and $Z_j=(vabd)$, so we may assume that we have
$Z_i=(vabc)$ and $Z_j=(vbad)$.  But then $a$ and $b$ are adjacent
neighbors of $v$, a contradiction.
 Therefore, $Z_i''$ and $Z_j''$ share no edges of $G-v$, and by
the hypothesis on $\Phi'$ they share at most one edge of
$G'-v'$, so $m_{22} \le 1$, and $\Phi''$ is face-simple.
 \end{proof}

 \subsection{Graphical surfaces}
 \label{ss:graphical}

 White \cite{Wh72} showed that any composition $\double{G}$, where $G$
is a simple graph without isolated vertices, has an orientable
quadrangular embedding.  Craft \cite{DC91,DC98} developed graphical
surfaces, which yield a simple proof of this result.  We outline his
proof, since we need his construction in Section \ref{sec:embgsvg}.

 For a graph $G$, the {\em graphical surface $S(G)$} derived from $G$ is
a surface obtained from an embedding of $G$ in $\mR^3$ by blowing up
every vertex $u$ into a sphere $\Si_u$ and replacing every edge $uv$ by
a tube $T_{uv}$ joining the spheres $\Si_u$ and $\Si_v$.  Since we work
in $\mR^3$, the resulting surface $S(G)$ is orientable.

 \begin{lemma}[Craft \cite{DC91,DC98}]\label{lem:compk2ori}
 Let $G$ be a connected simple graph. Then
$G[\,\ovl{K_2}\,]$ has a quadrangular embedding in the graphical
surface $S(G)$.
 \end{lemma}

 \begin{proof}[Outline of proof]
 We embed $\double{G}$ in $S(G)$ as
follows.  For any vertex $u \in V(G)$, let $u_\cN$ (north pole) and
$u_\cS$ (south pole) be two points in the sphere $\Si_u$; they
represent the two vertices of $\double{G}$ corresponding to $u$.
 We may assume that all tubes joined to $\Si_u$ are joined in some
cyclic order around the equator of $\Si_u$.
 There are four edges of $\double{G}$ corresponding to each $uv
\in E(G)$, which are $u_\cN v_\cN, u_\cN v_\cS, u_\cS v_\cS$ and $u_\cS
v_\cN$.
 These can all be embedded along $T_{uv}$ (there are two different ways
to do this, but for our purposes it will not matter which is used).
 In the resulting embedding, every edge is contained in two quadrilaterals.
 For example, $u_\cS v_\cN$ is contained in quadrilaterals
$Q_u=(t_\vX u_\cN v_\cN u_\cS)$ and $Q_v=(u_\cS v_\cN w_\vY v_\cS)$ where
$tu, vw \in E(G)$ and
$\vX, \vY \in \{\cN, \cS\}$.
 Note that if $u$ has degree $1$, then $t_\vX = v_\cS$.
 \end{proof}

 White and Craft dealt only with orientable embeddings.  However, we can
also produce nonorientable embeddings.
 Given a graphical surface, we can replace a tube $T_{uv}$ by a {\em
twisted tube\/} $\tw{T}_{uv}$ by taking a simple closed curve $\ga$
around $T_{uv}$ with a specified positive direction, cutting along it to
produce two boundary curves $\ga_1$ and $\ga_2$, then re-identifying
$\ga_1$ with $\ga_2$ so that the positive direction along $\ga_1$
corresponds to the negative direction along $\ga_2$ (this cannot be done
in $\mR^3$).  We can still embed the four edges between $\{u_\cN,
u_\cS\}$ and $\{v_\cN, v_\cS\}$ along $\tw{T}_{uv}$; they become
orientation-reversing edges relative to the original orientation at each
vertex.  Depending on which tubes we replace, the resulting embedding
may be nonorientable.

 \begin{lemma}\label{lem:compk2non}
 Let $G$ be a connected simple graph with at least one
cycle.  Then $\double{G}$ has a quadrangular embedding in a
nonorientable modified graphical surface $\tw{S}(G)$.
 \end{lemma}

 \begin{proof}
 Choose one edge $uv$ belonging to a cycle and replace the tube $T_{uv}$
by a twisted tube $\tw{T}_{uv}$ in the construction of Lemma
\ref{lem:compk2ori}.  The resulting embedding is nonorientable because
the cycle $(uvw \ldots z)$ in $G$ gives an orientation-reversing cycle
$(u_\cN v_\cN w_\cN \ldots z_\cN)$ in the embedding of $\double{G}$ in
the new surface $\tw{S}(G)$.
 \end{proof}


\subsection{Voltage graphs}

 We assume the reader is familiar with voltage graph constructions for
embeddings.  We summarize the main features; see \cite{GT} for more details.


 Given a graph $G$, assign an arbitrary {\em plus direction} to each
edge.
 A function $\al$ from the plus-directed edges of $G$ to a group $\Ga$
is an {\em ordinary voltage assignment} on $G$.
 The pair $\langle G, \al \rangle$ is called an {\em ordinary voltage
graph\/}.
 The {\em derived graph} $G^\al$ has vertex set $V(G)\times \Gamma$ and
an edge from $u_a=(u, a)$ to $v_b=(v,b)$ whenever $uv$ is a
plus-directed edge in $G$ and $b=a \cdot \alpha(uv)$.

 If $G$ has an embedding $\Phi$, represented by a rotation of edges at
each vertex and edge signatures, then $G^\al$ has a {\em derived
embedding $\Phi^\al$}:
 for each $u_a \in V(G^\al)$ use the natural bijection between edges
incident with $u$ in $G$ and edges incident with $u_a$ in $G^\al$ to
define the rotation at $u_a$ from the rotation at $u$, and give each
edge $u_av_b$ of $G^\al$ the signature of the corresponding edge $uv$ in
$G$.

 For each walk $W= v_0 e_1 v_1 e_2 v_2 \ldots e_k v_k$ in $G$ define its
{\em total voltage\/} to be $\al(e_1)^{\ep_1} \al(e_2)^{\ep_2} \ldots
\al(e_k)^{\ep_k}$ where $\ep_i$ is $+1$ if $W$ uses $e_i$ in the plus
direction and $-1$ otherwise.  
 The faces of $\Phi^\al$ come from the faces of $\Phi$: each face in
$\Phi$ with degree $k$ whose boundary walk has total voltage of order
$r$ in $\Ga$ yields $|\Ga|/r$ faces of degree $kr$ in $\Phi^\al$.
 Also, $\Phi^\al$ is nonorientable if and only if $\Phi$ is
nonorientable and has an orientation-reversing closed walk whose total
voltage is the identity of $\Ga$.

 \section{Embeddings from diamond sums} 
 \label{sec:embdiamsum}

 In this section we prove Theorems \ref{thm:genusori} and
\ref{thm:genusnon} by constructing embeddings of minimum genus with all face
degrees at least $4$, and with all face degrees even, for each complete
graph $K_n$, $n \ge 4$.
 Our constructions are inductive.
 The base cases are provided in \shortlong{an appendix}{Appendix
\ref{sec:small}}.

 The induction steps use quadrangular embeddings of complete bipartite
graphs and of $K_7^+$ and $K_{11}^+$, where $K_n^+$ denotes the graph
obtained from $K_n$ by subdividing an edge.
 In Figure \ref{fig:K7+} we provide embeddings $\tw\Psi_7$ of $K_7$ in
$N_5$ (as a polygon with labeled vertices, indicating how edges are to
be identified around the boundary) and $\Psi_{11}$ of $K_{11}$ in $S_9$
(as a rotation system; see \cite[Section 3.2]{GT}). 
 Each embedding is face-simple and all faces are $C_4$-faces apart from
two $C_3$-faces that share an edge $xy$ ($xy=01$ for $\tw\Psi_7$ and
$xy=56$ for $\Psi_{11}$).
 Nonorientability of $\tw\Psi_7$ follows from the fact that there are
edges, such as 05, used twice in the same direction around the outer
boundary of the polygon.
 By subdividing $xy$ with a vertex $z$ in each case, we obtain
embeddings $\tw\Psi_7^+$ and $\Psi_{11}^+$ of graphs $K_7^+$ and
$K_{11}^+$.  These embeddings are not face-simple, but with the choice
$v'=x$ they satisfy the hypotheses for $\Phi'$ in Lemma
\ref{lem:diamsumquad}.

 \begin{figure}[!hbtp]\refstepcounter{figure}\label{fig:K7+}
 \begin{center}

 \leavevmode%
 $\vcenter{\hbox{%
  \includegraphics[scale=1]{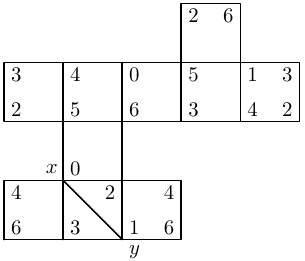}%
 }}$
 \hskip 20mm
 $\vcenter{\hbox{%
  \rotsys{%
   0.  1 3 4 6 a 9 8 7 5 2 \\
   1.  0 8 7 a 9 3 2 5 6 4 \\
   2.  0 1 5 6 4 3 9 a 7 8 \\
   3.  0 4 6 5 1 2 a 9 7 8 \\
   4.  0 8 6 5 7 9 a 2 3 1 \\
   \llap{$x=$ }%
   5.  0 9 a 4 6 2 3 7 8 1 \\
   \llap{$y=$ }%
   6.  0 1 8 7 3 2 5 4 a 9 \\
   7.  0 2 1 5 6 3 4 a 8 9 \\
   8.  0 7 9 a 4 3 6 5 1 2 \\
   9.  0 4 3 8 a 6 2 1 5 7 \\
   a.  0 9 7 5 1 2 6 8 3 4 \\
  }
 }}$

 {Figure~\ref{fig:K7+}: Embeddings $\tw\Psi_7$ of $K_7$ in $N_5$
(left) and $\Psi_{11}$ of $K_{11}$ in $S_9$ (right).}
 \end{center}
 \end{figure}

 As part of determining the orientable and nonorientable genera of
$K_{m,n}$, Ringel showed that $K_{m,n}$ has quadrangular embeddings in
certain cases.  Bouchet \cite{Bo78} later provided a simpler proof.

 \begin{lemma}[Ringel \cite{Ri65,Ri65-2}]\label{lem:Kmn}
 The complete bipartite graph $K_{m,n}$ has an orientable quadrangular
embedding whenever $(m-2)(n-2) \equiv 0 \pmod 4$ and $\min\{m,n\} \ge
2$,
 and a nonorientable quadrangular embedding whenever $mn \equiv 0 \pmod
2$ and $\min\{m,n\} \ge 3$.
 \end{lemma}

 First we consider nonorientable embeddings.  The following two lemmas
provide the induction steps in our proof.

 \begin{lemma}\label{lem:nonplus4}
 If a complete graph $K_n$, $n \ge 4$, admits a face-simple orientable
or nonorientable quadrangular embedding, then $K_{n+4}$ admits a
face-simple nonorientable quadrangular embedding.
 \end{lemma}

 \begin{proof}
 Take $K_7^+$ as described above, with special vertices $x, y, z$.
 We may interpret $K_7^+$ as a join $(K_1\cup K_5) + \ovl{K_2}$ where
$z$ is the vertex of the $K_1$, and $x$ and $y$ are the vertices of
the $\ovl{K_2}$.
 We build the embedding of $K_{n+4}$ in two steps from
 the embedding $\Phi_1=\tw\Psi_7^+$ of $K_7^+$ as described above,
 an orientable or nonorientable quadrangular embedding $\Phi_2$ of
$K_{6,n-1}$ from Lemma \ref{lem:Kmn},
 and the assumed quadrangular embedding $\Phi_3$ of $K_n$.

 Let $x'$ be a vertex of $K_{6,n-1}$ of degree 6. 
 Since $\delta(\Phi_2) = \min(6,n-1) \ge 3$, $K_{6,n-1}$ is bipartite, and using
Observation \ref{obs:FS}, we satisfy the hypotheses for $\Phi$ in Lemma
\ref{lem:diamsumquad} by taking $\Phi=\Phi_2$ and $v=x'$.
 From above we also satisfy the hypotheses for $\Phi'$ in Lemma
\ref{lem:diamsumquad} by taking $\Phi'=\Phi_1=\tw\Psi_7^+$ and $v'=x$.
 Therefore, by Lemma \ref{lem:diamsumquad}, applying the diamond sum to
$\Phi_1$ and $\Phi_2$ at $x$ and $x'$ yields a face-simple
quadrangulation $\Phi_{12}$, which is nonorientable since $\Phi_2$ is
nonorientable.
 The underlying graph
 $G_{12} = K_7^+\diamsum K_{6,n-1}$ is $(K_1\cup K_5)+\ovl{K_{n-1}}$,
where $z$ is the vertex of the $K_1$, with degree $n-1$, and $y$ is now
a vertex of the $\ovl{K_{n-1}}$.

 Since $\delta(\Phi_{12}) = \min(6,n-1) \ge 3$ and the neighbors of $z$
in $G_{12}$ are independent, we satisfy the hypotheses for $\Phi$ in
Lemma \ref{lem:diamsumquad} by taking $\Phi=\Phi_{12}$ and $v=z$.
 Let $z'$ be a vertex of $K_n$.
 We also satisfy the hypotheses for $\Phi'$ in Lemma
\ref{lem:diamsumquad} by taking $\Phi'=\Phi_3$ and $v'=z'$.
 Therefore, by Lemma \ref{lem:diamsumquad}, applying the diamond sum to
$\Phi_{12}$ and $\Phi_3$ at $z$ and $z'$ yields a face-simple
quadrangulation $\Phi_{123}$, which is nonorientable since $\Phi_{12}$
is nonorientable.
 The underlying graph $((K_1 \cup K_5)+\ovl{K_{n-1}}) \diamsum K_n$ is
$K_5+K_{n-1}=K_{n+4}$.
 \end{proof}

 The same proof also shows the following.  The $C_p$-face in the
embedding $\Phi_3$ of $K_n$ corresponds to a $C_p$-face in the embedding
$\Phi_{123} = \Phi_{12} \diamsum \Phi_3$ of $K_{n+4}$ by the
degree-preserving bijection of Lemma
\ref{lem:diamsumquad}.

 \begin{lemma}\label{lem:nonplus4gen}
 Suppose that $n \ge p \ge 5$.  If a complete graph $K_n$ admits a
face-simple orientable or nonorientable embedding in which all faces are
$C_4$-faces except for one $C_p$-face, then $K_{n+4}$ has a face-simple
nonorientable embedding in which all faces are $C_4$-faces except for
one $C_p$-face.
 \end{lemma}

 \begin{thm}\label{thm:detailnon}
 Given an integer $n$, let $k = 2+ \lceil n(n-5)/4\rceil$.

 Suppose that $n \ge 6$.
 If $n \equiv 0$ or $1 \pmod 4$ then $K_n$ has a face-simple
quadrangular embedding in $N_k$.
 If $n \equiv 2$ or $3 \pmod 4$ then $K_n$ has a face-simple embedding
in $N_k$ in which every face is a $C_4$-face except for one $C_6$-face.

 For $n=4$, $K_4$ has a quadrangular embedding in $N_k=N_1$ that is
closed-2-cell but not face-simple.  For $n=5$, $K_5$ has no quadrangular
embedding in $N_k=N_2$, but has an embedding in $N_3$ with three
$C_4$-faces and one $8$-face.
 \end{thm}

 \begin{proof}
 In each case the genus will follow by simple face/edge counting and
Euler's formula, so we focus on the other properties.
 For $n \not\equiv 1 \pmod 4$, \shortlong{the appendix}{Appendix
\ref{sec:small}} gives the required embeddings $\nk_n$ for $n \in
\{4,6,7,8\}$, and we then repeatedly apply Lemma \ref{lem:nonplus4} or
\ref{lem:nonplus4gen}.  (We need $n=8$ because the embedding for $n=4$
is not face-simple.)

 Suppose that $n \equiv 1 \pmod 4$.  For $n=5$, Hutchinson \cite[proof
of Theorem 2]{H75} and Hartsfield and Ringel \cite[Theorem 2]{HRnon}
showed that there is no quadrangular embedding of $K_5$ in the Klein
bottle $N_2$.
  However, there is a face-simple quadrangular embedding $\tw\Psi_6^-$
of $K_6-e$ in $N_3$ given in \shortlong{the appendix}{Appendix
\ref{sec:small}}, and deleting vertex $0$ gives the required embedding
of $K_5$ in $N_3$.
 For $n \ge 9$, applying Lemma \ref{lem:nonplus4} to the orientable
quadrangular embedding $\ok_5$ of $K_5$ from \shortlong{the
appendix}{Appendix \ref{sec:small}} gives a nonorientable quadrangular
embedding of $K_9$, and we then repeatedly apply Lemma
\ref{lem:nonplus4}.
 \end{proof}

 Theorem \ref{thm:genusnon} follows because every embedding
given in Theorem \ref{thm:detailnon} is even-faced.

 By adding chords (carefully, for $K_5$) or a single vertex inside the
face of degree greater than $4$ we also obtain the following\shortonly{
(see \cite{LLCEHYYZp} for details for $K_5$)}.

 \begin{corollary}\label{cor:quadnon}
 Suppose that $n \ge 4$ and $k = \lceil n(n-5)/4\rceil+2$.
 If $n \equiv 0$ or $1 \pmod 4$ and $n \ne 5$, then $K_n$ has a
quadrangular embedding in $N_k$, which is face-simple if $n \ge 8$.
 For $n=5$, $K_5$ is a subgraph of a quadrangular $5$-vertex embedding
with multiple edges in $N_3$, and of a face-simple simple $6$-vertex
quadrangular embedding in $N_3$.
 If $n \equiv 2$ or $3 \pmod 4$, $K_n$ is a subgraph of a quadrangular
$n$-vertex embedding with multiple edges in $N_k$, and of a face-simple
simple $(n+1)$-vertex quadrangular embedding in $N_k$.
 \end{corollary}

 \longonly{%
 \begin{proof}
 For $n=5$ we take $\tw\Psi_6^-$ from \shortlong{the appendix}{Appendix
\ref{sec:small}} as the simple $6$-vertex embedding.  Deleting vertex
$0$ from $\tw\Psi_6^-$ leaves an $8$-face $(1_1 4_1 5_1 2 1_2 3 4_2
5_2)$ (subscripting occurrences of the same vertex to distinguish them).
 We can add multiple edges $1_1 3$ and $5_1 3$.
 \end{proof}
 }

 Now we turn to orientable embeddings.
 The following two lemmas provide the induction steps in our proof. 
They are proved in exactly the same way as Lemmas \ref{lem:nonplus4} and
\ref{lem:nonplus4gen}, except that we take $\Phi_1$ to be the orientable
quadrangular embedding $\Psi_{11}^+$ of $K_{11}^+$ instead of the
nonorientable embedding $\tw\Psi_7^+$ of $K_7^+$, and $\Phi_2$ to be an
orientable quadrangular embedding of $K_{10,n-1}$, instead of a
nonorientable embedding of $K_{6,n-1}$.

 \begin{lemma}\label{lem:oriplus4}
 If a complete graph $K_n$, $n \ge 4$, admits a face-simple orientable
quadrangular embedding, then $K_{n+8}$ admits a face-simple orientable
quadrangular embedding.
 \end{lemma}

 \begin{lemma}\label{lem:oriplus4gen}
 Suppose that $5 \le p \le n$.  If a complete graph $K_n$ admits a
face-simple orientable embedding in which all faces are $C_4$-faces
except for one $C_p$-face, then $K_{n+8}$ has a face-simple orientable
embedding in which all faces are $C_4$-faces except for one $C_p$-face.
 \end{lemma}

 There is one case where we cannot find an embedding with all of the
properties we would like.

 \begin{prop}\label{prop:K6}
 Every general embedding of $K_6$ in $S_2$ is cellular with five
$C_4$-faces and two $C_5$-faces, and such an embedding exists.
 Thus, $K_6$ has no general even-faced embedding in $S_2$.
 \end{prop}

 \begin{proof}[Outline of proof]
 Let $\Phi$ be a general embedding of $K_6$ in $S_2$, with $n=6$ vertices,
$m=15$ edges, $r$ faces, and $r_i$ faces of degree $i$.
 By Euler's inequality, $r \ge \ec - n + m = -2-6+15 = 7$.
 By Observation \ref{obs:F2}, $\delta^*(\Phi) \ge 4$, and so
 $30=2m=4r_{4} + 5r_{5} + 6r_{6} + \ldots \ge 4r$.  Hence $r = 7$, with
either $r_{4}=6$ and $r_{6}=1$, or $r_4=5$ and $r_5=2$.  Since $r=7$,
$\Phi$ is cellular.  By Observation \ref{obs:F} the $4$-faces and any
$5$-faces are bounded by cycles; any $6$-face is bounded by a cycle or a
`bowtie' walk $(abcade)$.


 Now that we have restricted the structure of $\Phi$, we can perform a
case analysis to show that the embedding is as described.
 \shortlong{The details may be found in \cite{LLCEHYYZp}.  }{Details may
be found in Appendix \ref{sec:k6}.  }%
 It is also easy to generate and check all rotation systems for $K_6$
(up to isomorphism) by computer.  We did this; it ran in less than a
minute.
 The embedding $\ok_6$ from \shortlong{the appendix}{Appendix
\ref{sec:small}} demonstrates existence.
 \end{proof}

 \begin{thm}\label{thm:detailori}
 Given an integer $n$, let $h = 1+\lceil n(n-5)/8\rceil$.

 Suppose that $n=5$ or $n \ge 7$.
 If $n \equiv 0$ or $5 \pmod 8$, then $K_n$ has a face-simple
quadrangular embedding in $S_h$.
 If $n \not\equiv 0$ and $5 \pmod 8$ then $K_n$ has a face-simple
embedding in which every face is a $C_4$-face except for one $C_p$-face,
where $p \in \{6,8,10\}$ (specifically, $p = 12 - (n(n-5) \bmod 8)$).

 For $n=4$, $K_4$ has an embedding in $S_h=S_1$ with one $C_4$-face and
one $8$-face.
 For $n=6$, $K_6$ has no even-faced embedding in $S_h=S_2$, but has an
embedding in $S_2$ with five $C_4$-faces and two $C_5$-faces, and an
embedding in $S_3$ with four $C_4$-faces and one $14$-face.
 \end{thm}

 \begin{proof}
 In each case the genus will follow by simple face/edge counting and
Euler's formula, so we focus on the other properties.
 For $n \ne 4$ and $6$, \shortlong{the appendix}{Appendix
\ref{sec:small}} gives the required embeddings $\ok_n$ of $K_n$ for $n
\in \{5, 7,\ab 8,\ab 9,\ab 10,\ab 11,\ab 12, 14\}$, covering all classes
modulo $8$, and we then repeatedly apply Lemma \ref{lem:oriplus4} or
\ref{lem:oriplus4gen}.
 For $n=4$, there cannot be an embedding with a $C_8$-face, but an
embedding $\ok_4$ with an $8$-face is given in \shortlong{the
appendix}{Appendix \ref{sec:small}}.

 For $n=6$, see Proposition \ref{prop:K6}.
 To obtain the embedding of $K_6$ in $S_3$ with a $14$-face, take
$\ok_6$ from \shortlong{the appendix}{Appendix \ref{sec:small}} and swap
the positions of $1$ and $2$ in the rotation of vertex $0$.
 This replaces face boundaries $(01234)$, $(03142)$
and $(0251)$ by a single closed walk $(01234025103142)$ of length $14$.
 \end{proof}

 Theorem \ref{thm:genusori} follows because every embedding given in
Theorem \ref{thm:detailori}, except for the embedding of $K_6$ in $S_2$,
is even-faced.

 By adding chords (carefully, for $K_4$ and $K_6$) or a single vertex
(or two vertices, for $K_6$) inside the face of degree greater than $4$
we also obtain the following\shortonly{ (see \cite{LLCEHYYZp} for
details for $K_4$ and $K_6$)}.

 \begin{corollary}\label{cor:quadori}
 Suppose that $n \ge 4$ and $h = \lceil n(n-5)/8\rceil+1$.
 If $n \equiv 0$ or $5 \pmod 8$, then $K_n$ has a
quadrangular embedding in $S_h$.
 If $n \not\equiv 0$ and $5 \pmod 8$ and $n \ne 6$, $K_n$ is a subgraph
of a quadrangular $n$-vertex embedding with multiple edges in
$S_h$, and of a simple $(n+1)$-vertex quadrangular embedding in $S_h$.
 For $n=6$, $K_6$ is a subgraph of a quadrangular $6$-vertex embedding
with multiple edges in $S_3$ and a simple $8$-vertex quadrangular
embedding in $S_3$.
 \end{corollary}

 \longonly{%
 \begin{proof}
 For $K_4$, the $8$-face in $\ok_4$ is $(0_1 1_1 2_1 3_1 1_2 0_2 3_2
2_2)$ (subscripting occurrences of the same vertex to distinguish them)
and we can add multiple edges $0_1 3_1, 1_2 2_2$.  Adding a new vertex
$4$ adjacent to $0_1, 2_1, 1_2, 3_2$ gives an embedding isomorphic to
the the quadrangular embedding $\ok_5$ of $K_5$ in $S_1$.

 For $K_6$, the $14$-face from the proof of Theorem \ref{thm:detailori} is
 $(0_1 1_1 2_1 3_1 4_1 0_2 2_2 5 1_2 0_3 3_2 1_3 4_2 2_3)$ and we can
add multiple edges
 $0_1 3_1,\ab 3_1 2_2,\ab 3_1 4_2,\ab 2_2 0_3,\ab 0_3 4_2$
 or new vertices $6$ adjacent to $2_3, 1_1, 3_1, 0_2, 5$ and then $7$
adjacent to $1_2, 3_2, 4_2, 6$.  (Or the $8$-vertex simple
quadrangulation of $S_3$ containing $K_7$ also contains $K_6$.)
 \end{proof}
 }

 \section{Proof of the Even Map Color Theorem}
 \label{sec:color}

 The Map Color Theorem says that for a graph embedding $\Phi$ in a
surface $\Si \ne S_0$, $\chi(\Phi)
\le H(\Si) = 
 \left\lfloor \left(7 + \sqrt{49-24\ec(\Si)}\right)/2 \right\rfloor$
(the {\em Heawood number\/} of $\Si$), which can be improved to
$\chi(\Phi) \le H(\Si)-1$ if $\Si=N_2$, and these bounds are sharp.
 Hutchinson showed that this can be significantly improved if the
embedding is even-faced.
 In this section we strengthen her bound in one case, and use the
embeddings constructed in Section \ref{sec:embdiamsum} to show that the
bounds are sharp.

 \begin{thm}[{Hutchinson \cite[Theorems 1 and 2 and Corollary 2]%
	{H75}}]\label{thm:H75}
 For a surface $\Si$ define $H\even(\Si) =
 \left\lfloor \left(5 + \sqrt{25-16\ec(\Si)}\right)/2 \right\rfloor$. 
If $\Phi$ is an even-faced graph embedding in a surface $\Si \ne S_0$
then $\chi(\Phi) \le H\even(\Si)$.
 If $\Si = N_2$ this can be improved to $\chi(\Phi) \le
H\even(N_2)-1=4$.
 These results are sharp when $\Si = N_1$, $N_2$ or $S_1$.
 \end{thm}

 Hutchinson's proof requires $\Phi$ to be cellular,
because she first proves a face-coloring result and applies that to the
dual $\Phi^*$; each vertex of $\Phi$ needs to correspond to a distinct
face of $\Phi^*$.
 One way to extend the result to general embeddings is by first applying
the following lemma.

 \begin{lemma}\label{lem:C}
 Suppose $\Phi$ is a general embedding.  Then we can construct a new
embedding $\Phi'$ in the same surface by adding edges, such that (a)
$\Phi'$ is cellular, (b)
each face of $\Phi$ corresponds to a distinct face of $\Phi'$, (c) if two
faces are adjacent in $\Phi$ then the corresponding faces are adjacent
in $\Phi'$, (d) if $\Phi$ is even-faced then so is $\Phi'$, and
(e) if $\Phi$ is even-vertexed then so is $\Phi'$.
 \end{lemma}

 \begin{proof}
 We can add edges inside faces to destroy each internal handle or
crosscap (run an edge along the handle or across the crosscap)
and connect different boundary components without creating any new
faces, satisfying (a)--(d).
 If we replace each new edge by two parallel edges bounding a $2$-face we
still satisfy (a)--(d) and also satisfy (e).  In particular, since new
edges have the same face on both sides, this does not violate (c).
 \end{proof}

 We can also prove Theorem \ref{thm:H75} directly for general
embeddings.  We outline a proof of the general inequality in Theorem
\ref{thm:H75}, based on translating and simplifying the proof of
\cite[Theorem 1]{H75}, as we need some details later.  We use the
following preliminary results, which are implicit in the arguments of
\cite{H75}\shortonly{; proofs of Lemmas \ref{lem:R} and \ref{lem:sqrt}
may be found in \cite{LLCEHYYZp}}.

 \begin{obs}\label{obs:X}
 If we remove edges or vertices from an even-faced embedding it
remains even-faced.
 \end{obs}

 \begin{lemma}\label{lem:R}
 If $\Psi$ is a general $n$-vertex embedding in a surface $\Si$ with
average
(or minimum) vertex degree at least $d$ and average (or minimum) face
degree at least $4$, then $n(d-4) \le -4\ec(\Si)$.
 \end{lemma}

 \longonly{%
 \begin{proof}
 Suppose $\Psi$ has $m$ edges and $r$ faces, and let $\ec = \ec(\Si)$.
 By Euler's inequality $n-m+r \ge \ec$, and the degree
conditions mean that $2m \ge nd$ and $2m \ge 4r$.  Thus,
 $$4 \ec \le 4n-4m+4r = 4n -2m + (4r-2m) \le 4n-nd+0 = n(4-d)$$
 and the result follows.
 \end{proof}
 }

 \begin{lemma}\label{lem:sqrt}
 Let $\Si$ be a surface with $\Si \ne S_0$, and let $d = H\even(\Si)$. 
Then $d \ge 4$ and $d$ is the smallest positive integer such that
$(d+1)(d-4) > -4\ec(\Si)$.
 \end{lemma}

 \longonly{%
 \begin{proof}
 Consider $p(x) = (x+1)(x-4)+4\ec(\Si)=x^2-3x-4+4\ec(\Si)$.  Since
$\ec(\Si) \le 1$, $p(1)=p(2)<0$, and so $p(x)$ has two real roots $\al_1
< 1$ and $\al_2 > 2$.
 By the quadratic formula, $\al_2 = (3+\sqrt{25-16\ec(\Si)})/2$.
 A positive integer $d$ has $p(d) > 0$ if and only if $d > \al_2$.
 But since $d$ is an integer, $d > \al_2$ is
equivalent to $d \ge \lfloor \al_2+1 \rfloor = H\even(\Si)$.
 So the smallest positive integer $d$ with $(d+1)(d-4) > -4\ec(\Si)$, or
$p(d) > 0$, or $d > \al_2$, is exactly $H\even(\Si)$.

  Since $\ec(\Si) \le 1$, if $d=H\even(\Si)$ then the formula for
$H\even(\Si)$ yields $d \ge 4$.
 \end{proof}
 }

 \begin{proof}[Outline of proof that $\chi(\Phi) \le H\even(\Si)$ for
general embeddings]
 We label the steps here for reference.  Let $\ec=\ec(\Si) \le 1$.
 (A)~Let $d = H\even(\Si) \ge 4$.
 (B) Let $\Phi$ be an even-faced embedding of minimum order $n$ in
$\Si$ whose graph $G$ is not $d$-colorable.
 (C)~If $n \le d$, then $G$ is $d$-colorable, so $n \ge d+1$.
 (D)~Remove all loops from $\Phi$ to obtain an even-faced
embedding $\Phi_1$ of $G_1$ with $\chi(G_1) = \chi(G)$.
 (E)~Starting with $\Phi_1$ repeatedly remove one edge from each
$2$-face until no $2$-faces remain, giving an even-faced embedding
$\Phi_2$ of a graph $G_2$ with $\chi(G_2) = \chi(G)$ and
$\delta^*(\Phi_2) \ge 4$.
 (F)~If $\delta(\Phi_2) \le d-1$, then we can remove a vertex $x$ of
degree at most $d-1$, $d$-color $G_2-x$ by minimality since $\Phi_2-x$
remains even-faced, then color $x$, so $G_2$, and hence $G$, is
$d$-colorable.  Thus, $\delta(\Phi_2) \ge d$.
 (G) By (E), (F), Lemma \ref{lem:R} and (C),
 $(d+1)(d-4) \le n(d-4) \le -4\ec$, contradicting Lemma \ref{lem:sqrt}.
 \end{proof}

 Working with general embeddings, rather than cellular embeddings,
simplifies the above proof in step (F): we do not have to worry about
losing cellularity when we delete a vertex.

 We now show that the bound of Theorem \ref{thm:H75} can be improved
by $1$ when the surface is $S_2$, and then show that with this
improvement the result is sharp.

 \begin{prop}\label{prop:s2col}
 Suppose $G$ is a graph (multiple edges and loops allowed) with a
general even-faced embedding $\Phi$ in $S_2$.
 Then (ignoring loops when coloring) $\chi(G) \le H\even(S_2)-1 = 5$.
 \end{prop}

 \begin{proof}
 We modify the above proof that $\chi(\Phi) \le H\even(\Si)$.
 In (A) take $d=5$ instead of $d=H\even(S_2)=6$.
 As in (B), let $\Phi$ be an even-faced embedding of minimum order
$n$ in $\Si=S_2$ ($\ec=-2$) of a graph $G$ with $\chi(G) > d=5$.
 Apply Lemma \ref{lem:C} to convert $\Phi$ into an even-faced cellular
embedding $\Phi_0$ of a necessarily connected graph $G_0$.
 Apply (D) and (E) to $\Phi_0$ to construct $\Phi_1$ and $\Phi_2$ as
above.  Then $G_2$ is loopless and connected, $\chi(G_2) = \chi(G_1) =
\chi(G_0) \ge \chi(G)$, and $\delta^*(\Phi_2) \ge 4$.
 From (C) and (F), $n \ge d+1=6$ and $\delta(\Phi_2) \ge d = 5$.

 Delete edges from $\Phi_2$ to obtain an even-faced
embedding $\Phi_3$ in $S_2$ of an underlying connected simple graph
$G_3$ of $G_2$.  We have $\chi(G_3) = \chi(G_2) \ge \chi(G) > 5$.
 By the argument of (F), we still have $\delta(\Phi_3) \ge 5$.
 Let $m$ and $r$ be the number of edges and faces of $\Phi_3$,
respectively.

 Suppose that $n=6$.  If $G_3 \ne K_6$, then $G_3$, and hence $G$, is
$5$-colorable, so $G_3 = K_6$.  But then, by Proposition \ref{prop:K6},
$\Phi_3$ cannot exist, a contradiction.

 So $n \ge 7$.
 If $\Delta(G_3) = 5$, then since $G_3$ is
simple, connected and not equal to $K_6$, by Brooks' Theorem, $G_3$, and
hence $G$, is $5$-colorable. 
 Thus, $G_3$ has a vertex of degree $6$ or more, from which $2m \ge
5(n-1)+6 = 5n+1$, so $m \ge \lceil (5n+1)/2 \rceil$.
 By Observation \ref{obs:F2}, $\delta^*(\Phi_3) \ge 4$, so
$2m \ge 4r$ and $r \le m/2$.  Therefore, $-2 = \ec \le n-m+r
\le n-m+m/2 = n-m/2 \le n - \lceil (5n+1)/2 \rceil/2$.  This fails
if $n \ge 8$, so $n = 7$.  Moreover, when $n=7$ this is tight, so all
steps in our reasoning are tight.  In particular, $G_3$ has one vertex
of degree $6$ and $n-1=6$ vertices of degree $5$.  But then $G_3 = K_7 -
3K_2$ (delete three independent edges from $K_7$), which is
$4$-colorable, a contradiction.

 So no such $\Phi$ exists, and $\chi(G) \le 5$ for all $G$ with an even-faced
embedding in $S_2$.
 \end{proof}

 \begin{figure}[tbp]\refstepcounter{figure}\label{fig:addhc}
 \begin{center}

  \includegraphics[scale=0.9]{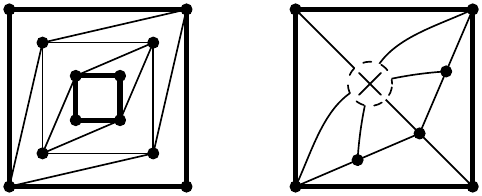}

  {Figure~\ref{fig:addhc}: Adding a quadrangular handle (left) or
crosscap (right).}

 \end{center}
 \end{figure}

 Before proving the main result of this section, we introduce two
operations.  By {\em adding a quadrangular handle} to an orientable
quadrangulation $\Phi$, we mean deleting two distinct faces, and
inserting a handle (cylinder or annulus) with four new vertices as shown
at left in Figure \ref{fig:addhc}, identifying the inner and outer
$4$-cycles with the boundaries of the deleted faces, so that the
resulting quadrangular embedding $\Phi'$ is still orientable.  Note that
if $\Phi$ is simple, so is $\Phi'$, and if $\Phi$ is face-simple, so is
$\Phi'$.

 Also, by {\em adding a quadrangular crosscap} to a quadrangular
embedding $\Phi$ we mean inserting a crosscap and three new vertices in
a face, as shown at right in Figure \ref{fig:addhc}.  The resulting
quadrangular embedding $\Phi'$ is nonorientable.  If $\Phi$ is simple,
so is $\Phi'$, and if $\Phi$ is face-simple, so is $\Phi'$.

 In both cases, $\chi(\Phi') \ge \chi(\Phi)$.

 \begin{proof}[Proof of Theorem \ref{thm:evenmct}, the Even Map Color
Theorem]
 Consider first the vertex-coloring result (a).  The upper bound on
$\chi(\Phi)$ follows from Hutchinson's Theorem \ref{thm:H75}, our
Proposition \ref{prop:s2col}, and the well-known fact that graphs with
even-faced embeddings in the plane are bipartite.
 Hutchinson provided sharpness examples for $\Si = N_1, N_2$ and $S_1$,
but we now provide sharp quadrangular examples for all surfaces.

 First we apply the results of Section \ref{sec:embdiamsum}.
 For $n \ge 7$, it follows from Corollaries \ref{cor:quadnon} and
\ref{cor:quadori} and addition of quadrangular handles or quadrangular
crosscaps that there is a face-simple quadrangular embedding
$\Omega_{n,\Si}$ in $\Si$ of a simple graph with $K_n$ as a subgraph,
and hence with $\chi(\Omega_{n, \Si}) \ge n$, provided $\ec(\Si) \le
n(5-n)/4$, i.e., $n(n-5) \le -4\ec(\Si)$.  This also works for $n=5$ if
$\Si$ is orientable and for $n=6$ if $\Si$ is nonorientable.  

 Now suppose that $\ec(\Si) \le -4$, and let $d=c(\Si) = H\even(\Si) \ge
7$.  
 By Lemma \ref{lem:sqrt}, $d$ is the smallest positive integer with
$(d+1)(d-4) > -4\ec(\Si)$, so this inequality fails with $d$ replaced by
$d-1$, giving $d(d-5) \le -4\ec(\Si)$.  Hence, by the previous
paragraph, there is an embedding $\Omega_{d,\Si}$ with $d \le
\chi(\Omega_{d,\Si}) \le H\even(\Si)=d$, which provides the required
sharpness example.

 The remaining surfaces are $S_h$ for $0 \le h \le 2$ and $N_k$ for $1
\le k \le 5$.
 For $S_0$ take the standard planar (spherical) embedding of the cube.
 If $\Si = S_1$ or $S_2$ take $\Omega_{5,\Si}$, which has
$5 \le \chi(\Omega_{5,\Si}) \le c(\Si) = 5$.
 For $N_1$ and $N_2$, take $\nk_4$ in $N_1$ (from \shortlong{the
appendix}{Appendix \ref{sec:small}}), add four new vertices inside each
face as shown at left in Figure \ref{fig:addpl} to obtain face-simple
$\nk_4'$ in $N_1$, and then add a quadrangular crosscap to give
$\nk_4''$ in $N_2$.
 Then $4=\chi(\nk_4) \le \chi(\nk_4') \le \chi(\nk_4'') \le
c(N_1)=c(N_2)=4$, so take $\nk_4'$ and $\nk_4''$ for $N_1$ and $N_2$,
respectively.
 For $N_3$ add a quadrangular crosscap to $\ok_5$ (from \shortlong{the
appendix}{Appendix \ref{sec:small}}).
 If $\Si=N_4$ or $N_5$ use $\Omega_{6,\Si}$.

 \begin{figure}[tbp]\refstepcounter{figure}\label{fig:addpl}
 \begin{center}

  \includegraphics[scale=0.9]{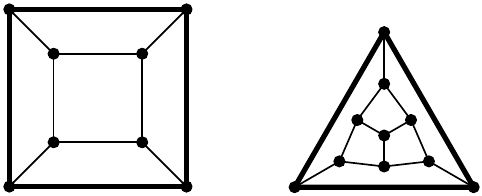}

  {Figure~\ref{fig:addpl}: Adding new faces in the same surface.}

 \end{center}
 \end{figure}

 The face-coloring version (b) follows from the vertex coloring-version
(a) by taking duals, after applying Lemma \ref{lem:C} if an embedding is
noncellular.
 All of the vertex-coloring sharpness examples are quadrangular
(implying closed-$2$-cell), face-simple and simple, so their duals
provide face-coloring sharpness examples that are $4$-regular,
closed-$2$-cell, simple and face-simple.
 \end{proof}

 We cannot extend the Even Map Color Theorem to embeddings with all face
degrees at least $4$, because there is no counterpart to Observation
\ref{obs:X} for such embeddings.
 They may realize the Heawood bound of the original Map Color Theorem.
 Suppose $\Phi$ is a sharpness example (such as a triangular embedding
of some $K_n$) for the original Map Color Theorem in a surface $\Si \ne
S_0$, $N_2$, so that $\chi(\Phi) = H(\Si)$.
 If we add new vertices inside each triangular face as shown at right in
Figure \ref{fig:addpl}, we obtain an embedding $\Phi'$ in $\Si$ with
$\delta^*(\Phi') \ge 4$ and $\chi(\Phi') = H(\Si) > H\even(\Si)$.
 \section{Embeddings from graphical surfaces and voltage graphs}
 \label{sec:embgsvg}

 In this section, we use graphical surfaces and voltage graphs to
construct both orientable and nonorientable quadrangular embeddings of
certain graphs of the form $G[K_4]$, proving Theorem \ref{thm:compk4}.

 \begin{thm}\label{thm:compk4ori}
 Let  $G$ be a connected simple graph with a perfect matching. Then
$G[K_4]$ has a face-simple orientable quadrangular embedding.
 \end{thm}

 \begin{proof} Let $G$ have perfect matching $M$, and let $S(G)$ be the
graphical surface derived from $G$.

 First, construct a quadrangular embedding $\eH$ of $H=\double{G}$
in $S(G)$ as in Lemma~\ref{lem:compk2ori}.
 For each vertex $v \in V(G)$ there are two vertices $v_\cN, v_\cS \in
V(H)$.
 For each $uv \in E(G)$, there is a tube $T_{uv}$ in $S(G)$, along which
run the edges  $u_\cN v_\cN, u_\cN v_\cS, u_\cS v_\cS$ and $u_\cS v_\cN$
of $H$.
 Each edge $u_\vP v_\vQ$ of $H$ belongs to two quadrilaterals of the
form $(u_\cN v_\vQ u_\cS t_\vX)$ and $(u_\vP v_\cN w_\vY v_\cS)$ where
$tu, vw \in E(G)$ and
$\vP, \vQ, \vX, \vY \in \{\cN, \cS\}$.

 Modify the embedding $\eH$ by splitting each edge into a digon
($2$-cycle) bounding a face.  Let $\eJ$ be the new embedding, with
underlying graph $J$.
 The other faces of $\eJ$ are quadrilaterals, in
one-to-one correspondence with the quadrilaterals of $\eH$.
 We now assign voltages from the group $\mZ_2$; since all elements of
$\mZ_2$ are self-inverse, the designation of plus directions for edges
does not matter.
 Choose a voltage assignment $\alpha: E(J)\to \mZ_2$ so that the
voltages of the edges of $J$ around each tube $T_{uv}$ alternate between
$0$ and $1$.
 Then each digon of $\eJ$ has one edge of voltage $0$ and one edge of voltage
$1$.  Each quadrilateral of $\eJ$, which uses edges from two tubes
$T_{uv}$ and $T_{vw}$, has an edge of voltage $0$ and an edge of voltage
$1$ on $T_{uv}$, and similarly for $T_{vw}$.  Therefore, every digon has
total voltage $1$ and every quadrilateral has total voltage $0$ in
$\langle J, \alpha\rangle$.
 Thus, the derived embedding $\eJ^\al$, with underlying graph $J^\al =
\double{H} = \double{\double{G}} = G[\,\ovl{K_4}\,]$, is an orientable
quadrangulation.

 We could also have obtained a quadrangular embedding of
$G[\,\ovl{K_4}\,] = \double{\double{G}}$ directly from Lemma
\ref{lem:compk2ori}, but that would not have had the special structure which
we now exploit to obtain an embedding of $G[K_4]$.
 We work with the vertices of $G$ in pairs specified by the perfect
matching $M$.


 \begin{figure}[!hbtp] \refstepcounter{figure}\label{fig:gs}
 \begin{center}
 \includegraphics[scale=1]{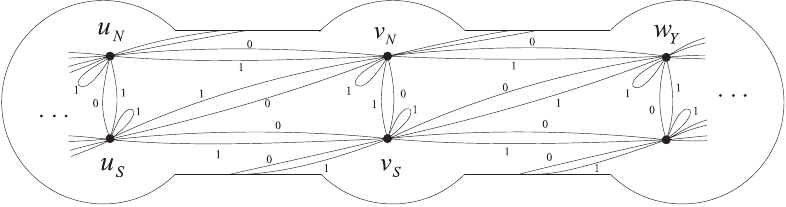}\\
 {Figure~\ref{fig:gs}: Voltage graph $\langle J_1, \alpha_1 \rangle$
generated from graphical surface of $\double{G}$.}
 \end{center}
 \end{figure}

 For each $uv \in M$, let $e$ be one of the four edges of $H$ on
the tube $T_{uv}$, between $\{u_\cN, u_\cS \}$ and $\{v_\cN, v_\cS \}$.
 We choose $e = u_\cS v_\cN$, as this makes it easier to illustrate what
is happening (see Figure~\ref{fig:gs}).
 In $\eH$, $e=u_\cS v_\cN$ belongs to two quadrilaterals $Q_u=(t_\vX
u_\cN v_\cN u_\cS)$ and $Q_v=(u_\cS v_\cN w_\vY v_\cS)$ where $tu, vw
\in E(G)$ and $\vX, \vY \in \{\cN, \cS\}$.
 Let the two edges of the digon in $J$ corresponding to $e$ be $e_1$ and
$e_2$, where $e_1$ belongs to the quadrilateral $Q'_u$ of $\eJ$
corresponding to $Q_u$ and $e_2$ belongs to the quadrilateral $Q'_v$ of
$\eJ$ corresponding to $Q_v$.
 Add a digon of two edges $d_1$ and $d_2$ in $Q_u'$ between $u_\cN$ and
$u_\cS$, and a digon of two edges $d_3$ and $d_4$ in $Q_v'$ between
$v_\cN$ and $v_\cS$, so we have four triangles $T_1(u)=(t_\vX u_\cN
u_\cS)$ using $d_1$, $T_2(u)=(u_\cN v_\cN u_\cS)$ using $d_2$,
$T_1(v)=(u_\cS v_\cN v_\cS)$ using $d_3$ and $T_2(v)=(v_\cS  v_\cN 
w_\vY)$ using $d_4$. Assign voltage 1 to $d_2$ and $d_3$, and 0 to $d_1$
and $d_4$. Insert a loop with voltage 1 at each of $u_\cN, u_\cS, v_\cN$
and $v_\cS$ and put these loops in the four different triangles $T_1(u),
T_2(u), T_1(v)$ and $T_2(v)$, respectively.

 Everything up to this point could have been done using independently
chosen quadrilaterals $Q'_u$ containing $u_\cN$ and $u_\cS$ and $Q'_v$
containing $v_\cN$ and $v_\cS$.
 However, the total voltages for the $4$-faces containing the
loops at $u_\cS$ and $v_\cN$ are currently $1$, so they will not
generate quadrilaterals in the derived embedding.
 To fix this, swap the voltages on $e_1$ and $e_2$: this is where we use
the pairing of vertices via $M$.
 Let $\eJ_1$, $J_1$ and $\al_1$ be the final embedding, graph and
voltage assignment, as shown in Figure \ref{fig:gs}.

 In $\eJ_1$ there are four types of faces.  Each $2$-face has total
voltage $1$ in $\langle J_1, \al_1 \rangle$, and each quadrilateral and
$4$-face containing a loop has total voltage $0$.  These three types of
faces all lift to quadrilaterals in $\eJ_1^{\al_1}$.
 The final type of face is bounded by a loop of total voltage $1$.
 This lifts to a face in $\eJ_1^{\al_1}$ bounded by a digon between
$(u_\vX,0)$ and $(u_\vX,1)$, where $u \in V(G)$ and $\vX \in \{\cN,
\cS\}$.
 Replacing each such digon in
$\eJ_1^{\al_1}$ by a single edge generates the required orientable
quadrangular embedding of $G[K_4]$, which is automatically face-simple
by Observation \ref{obs:FS}.
 \end{proof}

 As a special case of Theorem \ref{thm:compk4ori}, by taking $G=K_{2k}$ for
$k \ge 1$ we obtain a proof of Theorem \ref{thm:genusori} if the case
where $n \equiv 0 \pmod 8$.
 We can also obtain a nonorientable version of Theorem \ref{thm:compk4ori},
 which provides a proof of Theorem \ref{thm:genusnon} in the case where
$n \ge 16$ and $n \equiv 0 \pmod 8$.

 \begin{thm}\label{thm:compk4non}
 Let $G$ be a connected simple graph with a perfect matching and a
cycle. Then $G[K_4]$ has a face-simple nonorientable quadrangular
embedding.
 \end{thm}

 \begin{proof}
 Use Lemma \ref{lem:compk2non} instead of Lemma \ref{lem:compk2ori}
in the proof of Theorem \ref{thm:compk4ori}.  Replacing one or more tubes
by twisted tubes does not affect the argument.
 Take the orientation-reversing cycle $C=(u_\cN v_\cN w_\cN \ldots
z_\cN)$ in $H=\double{G}$ from the proof of Lemma \ref{lem:compk2non}
and replace each edge of $C$ by the edge of voltage $0$ in the
corresponding digon of $J_1$.  This gives an orientation-reversing cycle
of total voltage $0$ in $\langle J_1, \al_1 \rangle$, so the final
embedding is nonorientable.

 In the nonorientable case we also need to verify that the embedding is
face-simple.  We can properly $2$-face-color $\eJ_1$, coloring the
$4$-faces white and the other faces black; this lifts to a proper
$2$-face-coloring of $\eJ_1^{\al_1}$.
 In $\eJ_1$ each white face shares an edge with four distinct black
faces, so this also holds in $\eJ_1^{\al_1}$, and thus
$\eJ_1^{\al_1}$ is face-simple.
 Replacing digons makes each white face share at most one edge with
another white face, and the final embedding is still face-simple.
 \end{proof}

 Theorems \ref{thm:compk4ori} and \ref{thm:compk4non} together prove
Theorem \ref{thm:compk4}.
 \section{Minimal quadrangulations}
 \label{sec:minquad}

 In this section we apply our results to determine the order of some
minimal quadrangulations.

 Hartsfield and Ringel \cite{HRori,HRnon} showed that an $n$-vertex simple
quadrangulation of $\Si$ must satisfy $n(n-5) \ge -4\ec(\Si)$.
 They used this to investigate minimal quadrangulations of surfaces of
small genus, and to show that quadrangular embeddings of complete graphs
and generalized octahedra $O_{2k}=K_k[\overline{K_2}]$, $k \ge 4$ are
minimal.
 Lawrencenko \cite{L13} showed that certain orientable quadrangular
embeddings of a graph $\double{G}$, as described in Subsection
\ref{ss:graphical}, are minimal.
 The following lemma implies the minimality results
 of \cite{HRori, HRnon, L13}.

 \begin{lemma}\label{lem:minquad}
 Suppose that $L$ is obtained by deleting at most $n-4$ edges from the
complete graph $K_n$, $n \ge 5$.  Then any quadrangular embedding of $L$
is minimal.
 \end{lemma}

 \begin{proof}
 Let $f(x) = x(x-5)/2$.
 Suppose that $x \ge 5$.
 If $2\frac12 \le x' \le x-1$, then because $f$ is increasing on
$[2\frac12,\infty)$ we have $f(x)-f(x') \ge f(x)-f(x-1) = x-3$.
 If $1 \le x' \le 2\frac12$, then because $2\frac12 \le 5-x' \le x-1$ we
have $f(x)-f(x') = f(x)-f(5-x') \ge x-3$.
 Thus, $f(x)-f(x') \ge x-3$ whenever $1 \le x' \le x-1$.
 If $n$ is a nonnegative integer then $f(n) = \binom{n}{2} - 2n$.

 Now suppose $L$ has $n$ vertices, $m$ edges, and
a quadrangular embedding in $\Si$.
 Since at most $n-4$ edges of $K_n$ were deleted, $L$, and hence also
$\Si$, is connected.
 If we have another quadrangulation of $\Si$ with $n' \le n-1$ vertices
and $m'$ edges, then, since $m' = 2n'-2\ec(\Si)$ from Observation
\ref{obs:quadedges},
 \begin{eqnarray*}
 m' - \binom{n'}{2} &=& 2n' - 2\ec(\Si) - \binom{n'}{2} = -f(n') - 2\ec(\Si)
	= -f(n') + m - 2n \\
 &\ge& -f(n') + \binom{n}{2} -(n-4) - 2n = f(n)-f(n') - (n-4) \ge 1,
 \end{eqnarray*}
 proving that the other graph is not simple.
 \end{proof}

 Lemma \ref{lem:minquad} is sharp whenever $K_{n-1}$ has a quadrangular
embedding $\Phi$ of the appropriate orientability type (as in Theorems
\ref{thm:genusori} and \ref{thm:genusnon}).  Adding a new vertex of
degree $2$ adjacent to two opposite vertices of a face of $\Phi$ yields
a quadrangular embedding of a graph obtained from $K_n$ by deleting
$n-3$ edges, but this is not minimal.

 We can now apply Lemma \ref{lem:minquad} to Lemmas \ref{lem:compk2ori}
and \ref{lem:compk2non}, and to Theorems \ref{thm:compk4ori} and
\ref{thm:compk4non}.
 The orientable case of Corollary \ref{cor:minquadgs} is due to
Lawrencenko \cite[Theorem 2]{L13}.

 \begin{corollary}\label{cor:minquadgs}
 Let $k$ and $p$ be integers with $k \ge 4$ and $0 \le p \le k/4-1$.
 Suppose $G$ is obtained from $K_k$ by deleting $p$ edges.
 Then $\double{G}$ has both orientable and nonorientable quadrangular
embeddings that are minimal.
 Thus, minimal quadrangulations of the orientable surface of genus
 $k(k-3)/2-p+1$
 and of the nonorientable surface of genus
 $k^2-3k-2p+2$
 have order $2k$.
 \end{corollary}

 \begin{proof}
 Deleting $p$ edges from $K_k$ does not create isolated vertices or
destroy all cycles.
 Thus, by Lemmas \ref{lem:compk2ori} and \ref{lem:compk2non},
$\double{G}$ has orientable and nonorientable quadrangular embeddings.
 These have order $2k$, and are minimal by Lemma \ref{lem:minquad} since we
get $\double{G}$ by deleting $k+4p \le 2k-4$ edges from $K_{2k}$.
 We compute the genera of the surfaces from $m = 2n-2\ec(\Si)$.
 \end{proof}

 \begin{corollary}\label{cor:minquad}
 Let $\ell$ and $q$ be integers with $\ell \ge 1$ and $0 \le q \le
(\ell-1)/2$.
 Suppose $G$ is obtained from $K_{2\ell}$ by deleting $q$ edges.
 Then $G[K_4]$ has both orientable and nonorientable
quadrangular embeddings that are minimal.
 Thus, minimal quadrangulations of the orientable surface of genus
$8\ell^2-5\ell-4q+1$ and of the nonorientable surface of genus
$16\ell^2-10\ell-8q+2$ have order $8\ell$.
 \end{corollary}

 \begin{proof}
 If $\ell=1$ then $q=0$ and $G[K_4] = K_{8k}$, so orientable and
nonorientable quadrangular embeddings exist
by Theorems \ref{thm:genusori} and \ref{thm:genusnon}.
 If $\ell \ge 2$ then the $q$ edges deleted from $K_{2\ell}$ are
incident with at most $\ell-1$ vertices, so $G$ has
$K_{2\ell}-E(K_{\ell-1})$ as a subgraph, and hence has a perfect
matching and a cycle.
 Thus, by Theorems \ref{thm:compk4ori} and \ref{thm:compk4non}, $G[K_4]$
has the required embeddings.

 For all $\ell$ these embeddings have order $8\ell$, and are minimal
by Lemma \ref{lem:minquad} since we get $G[K_4]$ by deleting $16q <
8\ell-4$ edges from $K_{8\ell}$.
 We compute the genera of the surfaces from $m = 2n-2\ec(\Si)$.
 \end{proof}

 The simple quadrangulations described in Corollaries \ref{cor:quadnon}
and \ref{cor:quadori} are also minimal.
 The simple quadrangulations with $\ell = n+1$ vertices are embeddings
of $K_\ell$ with $\ell-4$, $\ell-5$ or $\ell-6$ edges deleted, and so
are minimal by Lemma \ref{lem:minquad}.

 \begin{corollary}\label{cor:minquadds}
 If $n \equiv 2$ or $3 \pmod 4$, $n \ge 6$ and $k =
2+\lceil n(n-5)/4 \rceil$, then a minimal quadrangulation of $N_k$ has
$n+1$ vertices. If $n \equiv 1$, $2$, $3$, $4$, $6$ or $7
\pmod 8$, $n \ge 7$ and $h = 1 + \lceil n(n-5)/8 \rceil$, then a minimal
quadrangulation of $S_h$ has $n+1$ vertices.
 \end{corollary}

 There is some overlap here between the conclusions about the order of
minimal quadrangulations.  The case of Corollary \ref{cor:minquad} with
$\ell/4 \le q \le (\ell-1)/2$ is also covered by Corollary
\ref{cor:minquadgs} with $k=4\ell$ and $p=4q-\ell$.  Some (but not all)
cases of Corollary \ref{cor:minquadds} are also covered by Corollary
\ref{cor:minquadgs}.

 \section{Conclusion}
 \label{sec:conclusion}

 We give some final remarks.

 \smallskip\noindent
 (1) Hartsfield and Ringel \cite{HRori,HRnon} defined quadrangulations
more strictly than we do: they insisted that two distinct faces share at
most one edge and at most three vertices.
 For an embedding of a simple graph, this is equivalent to being
face-simple.
 The reason for this restriction is unclear.  Perhaps they wished to
make the embedding ``polyhedral''.
 However, an embedding is now usually considered polyhedral if it is a
{\em $3$-representative\/} (every noncontractible simple closed curve in
the surface intersects the graph in at least three points) embedding of
a $3$-connected graph.
 A quadrangular embedding of $K_n$ is never polyhedral in this sense:
given a face $(uvwx)$, the edge $uw$ is part of the boundary of some
other face, and using these two faces we can find a simple closed curve
intersecting the graph at just $u$ and $w$, which must be
noncontractible.
 In any case, all our embeddings, with a few small exceptions, are
face-simple and so satisfy Hartsfield and Ringel's definition.

 \smallskip\noindent
 (2) It may be possible to carry out the graphical surface/voltage graph
construction from the proof of Theorem \ref{thm:compk4ori} with
non-perfect matchings $M$ of $G$ as well as with perfect matchings, to
give orientable and nonorientable quadrangular embeddings of some graphs
$L$ with $G[\,\ovl{K_4}\,] \subseteq L \subseteq G[K_4]$.
 This could provide some further examples of minimal quadrangulations.

 \smallskip\noindent
 (3) Our constructions have a lot of flexibility, particularly the
constructions from Subsection \ref{ss:graphical} and Section
\ref{sec:embgsvg}.  
 The graphical surface embeddings of $\double{G}$ in $S(G)$ (with
twisted tubes allowed) require a cyclic order of tubes around the
equator of each sphere, and a designation of which tubes are to be
twisted. (This corresponds to choosing an arbitrary embedding of $G$,
described by a rotation system with edge signatures.)
 There are two ways to run the edges along each tube.
 For Theorem \ref{thm:compk4ori} or \ref{thm:compk4non} we may choose an
arbitrary perfect matching $M$ of $G$, and  for each edge $uv$ of $M$ we
may choose one of four possible edges along the corresponding tube to
determine $Q_u$ and $Q_v$.
 We also have two ways to assign the voltages for the digons
of $J$ running along each tube. 

 It therefore seems natural to ask whether our techniques can be used to
provide useful lower bounds on the number of nonisomorphic quadrangular
embeddings of $K_n$.
 \appendix
 \section{Appendix: Small cases}
 \label{sec:small}

 In this appendix we provide the embeddings for the bases of the
inductive proofs in Section \ref{sec:embdiamsum}.

 \subsection{Nonorientable embeddings}
 \label{smallnon}

 At left in Figure \ref{fig:smallnon66} is a face-simple quadrangular
embedding $\tw\Psi_6^-$ of $K_6-e$ ($e=01$) in $N_3$, which is used for
constructing embeddings related to $K_5$.
 This is shown as a polygon with labeled vertices, indicating how edges
are to be identified around the boundary.  Nonorientability follows from
the existence of edges used twice in the same direction around the outer
boundary.

 We give nonorientable embeddings $\nk_n$ for $n \in \{4,6,7,8\}$ in
which all faces are $C_4$-faces except for possibly one $C_6$-face. 
They are all closed-2-cell and all except $\nk_4$ are face-simple.
 The embedding $\nk_4$ of $K_4$ is obtained by taking each of the three
hamilton $4$-cycles in $K_4$ as a face boundary.
 The embedding $\nk_6$ of $K_6$ with six $C_4$-faces and one $C_6$-face
is generated by the voltage graph shown at center in Figure
\ref{fig:smallnon66}.  The loop of voltage $3$ generates digons, which
are replaced by single edges.  A polygon representation of $\nk_6$ is
also given at right in Figure \ref{fig:smallnon66}.
 The embeddings $\nk_7$ of $K_7$ and $\nk_8$ of $K_8$ are shown at
left and right, respectively, in Figure \ref{fig:smallnon78}.

 \subsection{Orientable embeddings}
 \label{smallori}

 Below are orientable embeddings in which all faces are $C_4$-faces
except for at most two specified faces.
 These embeddings are represented as rotation systems using vertices 0,
1, 2, $\dots$, 9, a, b, c, d.  All embeddings are even-faced except for
$\ok_6$.  All are face-simple and closed-2-cell except $\ok_4$.

 \begin{figure}[tb]\refstepcounter{figure}\label{fig:smallnon66}
 \begin{center}

 \includegraphics[scale=0.9]{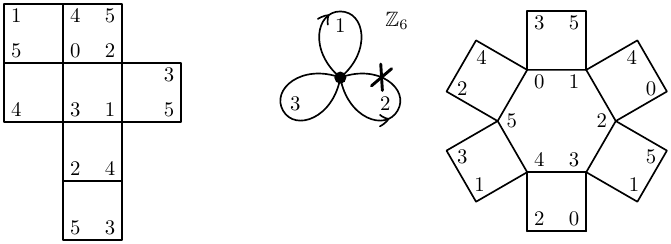}

 \smallskip
 {Figure~\ref{fig:smallnon66}: Nonorientable embeddings
 $\tw\Psi_6^-$ (left) and $\nk_6$ (center and right).}

 \end{center}
 \end{figure}

 \begin{figure}[tb]\refstepcounter{figure}\label{fig:smallnon78}
 \begin{center}

 \includegraphics[scale=0.9]{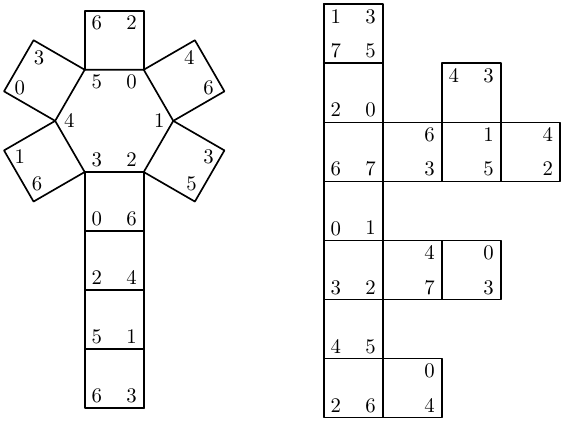}

 \smallskip
 {Figure~\ref{fig:smallnon78}: Nonorientable embeddings
 $\nk_7$ (left) and $\nk_8$ (right).}

 \end{center}
 \end{figure}


 \noindent
 \rotsys{%
  $\ok_4$ of $K_4$ \\
  \hnull\quad with $8$-face \\
  \hnull\quad $(0 1 2 3 1 0 3 2%
	)$: \\
  \noalign{\vskip3pt}
  0.  1 3 2 \\
  1.  0 2 3 \\
  2.  0 1 3 \\
  3.  0 2 1 \\
 }
 \hskip 12mm plus.01fil
 \rotsys{%
  $\ok_5$ of $K_5$: \\
  \noalign{\vskip3pt}
  0.  1 4 2 3 \\
  1.  0 4 3 2 \\
  2.  0 3 4 1 \\
  3.  0 1 2 4 \\
  4.  0 2 1 3 \\
 }
 \hskip 12mm plus.01fil
 %
 \rotsys{%
  $\ok_6$ of $K_6$ \\
  \hnull\quad with $C_5$-faces \\
  \hnull\quad $(0 1 2 3 4), (0 3 1 4 2%
	)$: \\
  \noalign{\vskip3pt}
  0.  1 2 3 5 4 \\
  1.  0 2 3 4 5 \\
  2.  0 5 1 3 4 \\
  3.  0 1 5 2 4 \\
  4.  0 1 2 5 3 \\
  5.  0 4 3 2 1 \\
 }
 \hskip 12mm plus.01fil
 %
 \rotsys{%
  $\ok_7$ of $K_7$ \\
  \hnull\quad with $C_6$-face \\
  \hnull\quad $(0 1 2 3 4 5%
	)$: \\
  \noalign{\vskip3pt}
  0.  1 4 3 2 6 5 \\
  1.  0 2 5 3 4 6 \\
  2.  0 6 1 3 5 4 \\
  3.  0 1 2 4 5 6 \\
  4.  0 6 2 3 5 1 \\
  5.  0 3 2 1 6 4 \\
  6.  0 4 1 5 2 3 \\
 }
 \hskip 12mm plus.01fil
 %
 \rotsys{%
  $\ok_8$ of $K_8$: \\
  \noalign{\vskip3pt}
  0.  1 7 4 3 5 2 6 \\
  1.  0 7 2 5 6 4 3 \\
  2.  0 7 3 4 6 1 5 \\
  3.  0 7 1 6 5 2 4 \\
  4.  0 6 1 2 5 3 7 \\
  5.  0 4 3 1 6 2 7 \\
  6.  0 5 2 1 3 4 7 \\
  7.  0 3 4 2 5 1 6 \\
 }
 \hskip 12mm plus.01fil
 %
 \rotsys{%
  $\ok_9$ of $K_9$ with $C_8$-face \\
  \hnull\quad $(0 1 2 3 4 5 6 7%
	)$: \\
  \noalign{\vskip3pt}
  0.  1 4 6 2 8 5 3 7 \\
  1.  0 2 6 4 3 7 5 8 \\
  2.  0 8 5 7 6 4 1 3 \\
  3.  0 8 2 4 6 1 7 5 \\
  4.  0 8 6 2 3 5 7 1 \\
  5.  0 7 3 2 1 4 6 8 \\
  6.  0 1 4 5 7 3 2 8 \\
  7.  0 5 8 2 3 4 1 6 \\
  8.  0 3 5 4 1 2 6 7 \\
 }
 \hskip 12mm plus.01fil
 %
 \rotsys{%
  $\ok_{10}$ of $K_{10}$ with $C_{10}$-face \\
  \hnull\quad $(0 1 2 3 4 5 6 7 8 9%
	)$: \\
  \noalign{\vskip3pt}
  0.  1 5 6 4 7 3 8 2 9 \\
  1.  0 2 4 3 6 7 5 8 9 \\
  2.  0 9 8 7 6 1 3 5 4 \\
  3.  0 9 1 8 2 4 5 7 6 \\
  4.  0 9 2 3 5 7 8 1 6 \\
  5.  0 9 4 6 7 1 3 2 8 \\
  6.  0 8 3 5 7 4 2 1 9 \\
  7.  0 6 8 5 3 1 2 4 9 \\
  8.  0 6 5 3 4 2 1 7 9 \\
  9.  0 4 6 3 7 5 1 2 8 \\
 }
 \hskip 12mm plus.01fil
 %
 \rotsys{%
  $\ok_{11}$ of $K_{11}$ with $C_{10}$-face \\
  \hnull\quad $(0 1 2 3 4 5 6 7 8 9%
	)$: \\
  \noalign{\vskip3pt}
  0.  1 5 7 3 a 6 4 8 2 9 \\
  1.  0 2 7 9 6 4 5 8 3 a \\
  2.  0 a 1 3 5 6 9 8 4 7 \\
  3.  0 a 7 2 4 1 9 8 6 5 \\
  4.  0 a 3 5 2 6 8 1 9 7 \\
  5.  0 a 3 2 7 4 6 1 8 9 \\
  6.  0 9 1 2 3 4 8 5 7 a \\
  7.  0 9 2 5 3 6 8 4 1 a \\
  8.  0 7 9 1 6 2 5 4 3 a \\
  9.  0 7 5 2 1 4 6 a 3 8 \\
  a.  0 5 1 4 6 3 7 2 8 9 \\
 }
 \hskip 12mm plus.01fil
 %
 \rotsys{%
  $\ok_{12}$ of $K_{12}$ with $C_{8}$-face \\
  \hnull\quad $(0 1 2 3 4 5 6 7%
	)$: \\
  \noalign{\vskip3pt}
  0.  1 a 2 9 3 8 4 b 6 5 7 \\
  1.  0 2 b 6 5 8 4 7 3 a 9 \\
  2.  0 b 7 5 1 3 8 4 6 9 a \\
  3.  0 b 7 8 5 9 6 2 4 1 a \\
  4.  0 b 1 a 3 5 8 2 9 6 7 \\
  5.  0 b 2 8 7 3 4 6 9 1 a \\
  6.  0 a 1 9 4 2 8 3 5 7 b \\
  7.  0 a 1 b 4 9 5 3 8 2 6 \\
  8.  0 a 1 9 2 6 4 3 5 7 b \\
  9.  0 a 1 2 4 8 5 3 7 6 b \\
  a.  0 9 2 1 4 8 3 7 5 6 b \\
  b.  0 7 4 8 3 9 1 5 6 2 a \\
 }
 \hskip 12mm plus.01fil
 %
 \rotsys{%
  $\ok_{14}$ of $K_{14}$ with $C_{6}$-face \\
  \hnull\quad $(0 1 2 3 4 5%
	)$: \\
  \noalign{\vskip3pt}
  0.  1 7 a 4 b 3 c 2 d 8 6 9 5 \\
  1.  0 2 6 7 a 3 b 9 4 8 5 c d \\
  2.  0 d c 1 3 9 b a 4 8 5 7 6 \\
  3.  0 d 2 4 6 9 1 c 5 7 8 a b \\
  4.  0 d 6 8 7 9 3 5 a 2 b 1 c \\
  5.  0 d 7 a 8 3 1 c 2 b 9 6 4 \\
  6.  0 d 2 a 5 3 9 4 8 7 b 1 c \\
  7.  0 d 6 9 1 b 4 3 8 5 a 2 c \\
  8.  0 c 1 b 2 a 7 5 3 4 9 6 d \\
  9.  0 c 1 a 7 8 6 4 b 3 5 2 d \\
  a.  0 c 1 9 b 2 3 5 8 4 6 7 d \\
  b.  0 c 3 2 5 8 4 a 1 9 7 6 d \\
  c.  0 b 4 9 6 2 1 3 a 7 5 8 d \\
  d.  0 6 8 4 a 5 9 3 b 7 1 2 c \\
 }

 \longonly{%
 
 \section{Appendix: Analysis of \texorpdfstring{$K_6$ in $S_2$}{K6 in S2}}
 \label{sec:k6}

 In this appendix we prove Proposition \ref{prop:K6}, and show
that $K_6$ has no even-faced general embedding in $S_2$.

 In an embedding the faces around a vertex must form a {\em proper
rotation}, a cyclic sequence closing up so that the vertex has a
neighborhood homeomorphic to an open disk.  If a potential set of faces
incident with a vertex $v$ close up in a cyclic sequence without including
all edges incident with $v$, we say the rotation at $v$ is {\em
improperly closed}.

 \begin{lemma}\label{lem:k6c6}
 $K_6$ has no orientable embedding in which every face is a $C_4$-face
except for one $C_6$-face.
 \end{lemma}

 \begin{proof}
 Assume that such an embedding exists.
 Since the embedding is orientable we may assign consistent orientations
to all the faces, and describe each face using a cyclic list of vertices
following the orientation (it is not equivalent to its reverse).

 We will label the vertices of $K_6$ by elements of
$\mZ_6=\{0,1,2,3,4,5\}$. 
 Each edge has an obviously defined {\it length} of $1$, $2$ or $3$
depending on $j-i$.
 Each arc (directed edge) from $i$ to $j$ has {\it length\/} $j-i \in
\mZ_6$ which we will write as an element of the set $\{-2,-1,1,2,3\}$
(where $-2=4$, $-1=5$).  
 Each arc is used exactly once by our embedding.
 Without loss of generality we may label the vertices so that the
$C_6$-face is $Z=(054321)$, using all the arcs of length $-1$.  Therefore
the remaining $6$ $C_4$-faces use 6 arcs of each length $1$, $2$, $-2$
and $3$.  Since no $C_4$ in $K_6$ can use more than two arcs of length
$\pm 2$, and altogether they use $12$ arcs of length $\pm 2$, each of
the $6$ $C_4$ faces must use exactly two arcs of length $\pm 2$.

 Therefore the cyclic pattern of lengths in each $C_4$-face (following
the arcs in their positive direction) must be one of $6$ possibilities:
$A=(1,1,2,2)$, $B=(1,2,1,2)$, $C=(1,-2,-2,3)$, $D=(1,3,-2,-2)$,
$E=(1,-2,3,-2)$ or $F=(2,3,-2,3)$. 
 For each pattern $P$ let $P_i$ be the potential face starting at vertex
$i$ and following pattern $P$; for example, $C_3 = (3420)$ ($3+1=4$,
$4-2=2$, $2-2=0$, $0+3=3$).

 Any face $A_i$ together with $Z$ improperly closes the rotation at
$i+1$, so there are no faces of pattern $A$.
 Let $n_F$ be the number of faces of pattern $F$ and $n_{CDE}$ the
number of faces of pattern $C$, $D$ or $E$.  Counting the arcs of length
$-2$ we have $2n_{CDE}+n_F=6$.  Counting the arcs of length $3$ we have
$n_{CDE}+2n_F=6$.  Therefore $n_{CDE}=n_F=2$.

 Consider the two faces of pattern $F$.  They must share an
edge of length $3$, which we may assume is $03$.  The faces of pattern
$F$ using this edge are $F_3$ and $F_4$, which use arc $03$, and $F_0$
and $F_1$, which use arc $30$.  We must have one face that uses arc $03$
and one that uses arc $30$.  Moreover, we cannot have $F_3$ and $F_0$
because they are reverses of each other, giving improper rotations at
all of their vertices if they occur together in an embedding. 
Similarly, we cannot have $F_4$ and $F_1$, because they are reverses of
each other.  So we must have $F_3$ and $F_1$, or $F_4$ and $F_0$. 
Without loss of generality we assume we have $F_4$ and $F_0$; if we have
$F_3$ and $F_1$ we just add $3$ to all the vertex labels and they become
$F_0$ and $F_4$.

 Now consider the arc $01$, which must belong to some face.
 The possible faces are $B_0=B_3$, $C_0$, $D_0$ or $E_0$.
 If $B_0$ is a face then the arc $45$ is used by both $B_0$ and $F_4$.
 If $C_0$ is a face then the arcs $53$ and $30$ are used by both $C_0$
and $F_0$, a contradiction.
 If $D_0$ is a face then the arc $14$ is used by both $D_0$ and $F_4$, a
contradiction.
 If $E_0$ is a face then $E_0$ and $F_0$ improperly close up the
rotation at vertex $2$, which is a contradiction.

 Hence all situations lead to a contradiction, so, as claimed, there is
no such embedding of $K_6$.
 \end{proof}

 \begin{lemma}\label{lem:k6q5}
 Every cellular orientable embedding of $K_6$ in which some vertex is
incident with five $C_4$-faces must have five
$C_4$-faces and two $C_5$-faces.
 \end{lemma}

 \begin{proof}
 Label the vertices of $K_6$ by $\infty$ and the elements of
$\mZ_5=\{0,1,2,3,4\}$.  Without loss of generality we may suppose that
$\infty$ has clockwise rotation $(0,1,2,3,4)$, and that for $i \in
\mZ_5$ there is a $4$-cycle face $(\infty,i,a_{i+3},i+1)$ (writing faces
also in clockwise order; this labelling makes $a_i$ `opposite' to $i$ in
the face neighborhood around $\infty$).  For each $i \in \mZ_5$ we must
have $a_i \notin \{\infty, i-3, i-2\}$, so $a_i \in \{i-1,i,i+1\}$ for
each $i$; let the rotation around $i$ be $(\infty, a_{i+2}, b_{i1},
b_{i2}, a_{i+3})$.  Since the same vertex cannot occur twice in the
rotation around $i-2$, $a_i \ne a_{i+1}$ for each $i$.

 Suppose first that $a_i = a_j$ for some $i \ne j$.  The only way this
can happen is if $j = i \pm 2$ and $a_i$ is the number between $i$ and
$j$.
 Without loss of generality suppose that $a_4=a_1=0$.  Then the rotation
around $0$ contains the sequence $4,3$ since $a_1=0$ and the sequence
$2,1$ since $a_4=0$.  Since $a_2 \ne 4$, the rotation around $0$ must be
$(\infty, a_2=2, 1, 4, a_3=3)$.  Since $a_3=3$, the rotation around $3$
contains the sequence $1,0$; but we already know that the rotation
around $3$ contains the sequence $b_{32}, a_1=0$ and hence $b_{32}=1$. 
This means that $a_0 \ne 1$.  By similar reasoning, $b_{21}=4$ and hence
$a_0 \ne 4$.  Also $a_0 \ne a_4 = 0$.  But now there are no possible
values for $a_0$, which is a contradiction.  So we know that $a_i \ne
a_j$ when $i \ne j$.  Hence each $j \in \mZ_5$ occurs exactly once as
some $a_i$.

 Suppose that $a_0=1$.
 The rotation around $a_0=1$ contains the
sequence $3,2$, so we have $(3,2) = (a_3,b_{11})$, $(b_{11},b_{12})$ or
$(b_{12},a_4)$.
 We cannot have $a_4=2$ so $(3,2)$ cannot be $(b_{12}, a_4)$. 
 If $(3,2)= (a_3,b_{11})$ then both $1=a_0$ and $3=a_3$ occur as vertices
$a_i$, so we must have $a_2=2$.  Then the rotation around $a_2=2$ contains
the sequence $0,4$.  Now $a_4 \ne a_3=3$ so $a_4 \in \{0,4\}$ and the
rotation around $2$ contains the sequence $\infty, a_4$, so we cannot
have $a_4=4$ and we must have $a_4=0$.  But now $1,2,0$ have all been
used as vertices $a_i$, so there is no valid value for $a_1$, a
contradiction.
 Hence we must have $(3,2)=(b_{11},b_{12})$.
 But then $a_3$ cannot be equal to either $3$ or $2$, so $a_3=4$.

 Generalizing the above, we have shown that $a_i = i+1$ implies that
$a_{i+3}=i+4$, and $b_{i+1,1}=i+3$, $b_{i+1,2}=i+2$.  Repeating this
reasoning determines the rotation around every vertex $i$ as being
$(\infty, i+3, i+2, i+1, i-1)$.  This gives an embedding of $K_6$ with
five $C_4$-faces and two $C_5$-faces $(03142)$ and $(01234)$, in
clockwise order.

 The case where $a_i  = i-1$ for some $i$ is symmetric.  So we need only
deal with the case where $a_i=i$ for all $i$.  However, this is
impossible: for example, it leads to the rotation around vertex $0$
containing $a_3,\infty=3,\infty$ but also the rotation around $a_0=0$
containing $3,2$.

 Thus, the only possible situations lead to the embedding specified.
 \end{proof}

 \begin{proof}[Proof of Proposition \ref{prop:K6}]
 Assume there is a general embedding $\Phi$ of $K_6$ in $S_2$.
 As noted in Section \ref{sec:embdiamsum}, $\Phi$ must be cellular with
six $C_4$-faces and one $6$-face, or five $C_4$-faces and two $C_5$-faces.

 Suppose there is a $6$-face.
 By Lemma \ref{lem:k6c6} the $6$-face is not a $C_6$-face.
 So the $6$-face has fewer than $6$ distinct vertices, and thus there is
some vertex all of whose adjacent faces are $C_4$-faces.  But then, by
Lemma \ref{lem:k6q5}, $\Phi$ does not have a $6$-face, a contradiction.

 Therefore the embedding has five $C_4$-faces and two $C_5$-faces.  The
embedding $\ok_6$ in Appendix \ref{sec:small}, or the embedding found in
the proof of Lemma \ref{lem:k6q5}, shows that such an embedding exists.
 \end{proof}

 }
 \section*{Acknowledgements}

 The authors thank Joan Hutchinson for information regarding
\cite{Ha94q}.

 Wenzhong Liu was partially supported by Fundamental Research Funds for
the Central Universities (NZ2015106) and NSFC grant (11471106).
 M.~N.~Ellingham was partially supported by National Security Agency
grant H98230--13--1--0233 and Simons Foundation award no.~429625.
 Dong Ye was partially supported by Simons Foundation award no.~359516.
 Xiaoya Zha was partially supported by National Security Agency grant
H98230--13--1--0216.

\end{document}